[1994/06/01]

\documentclass[12pt,reqno,intlimits,english,centertags]{amsart}

\usepackage[T1]{fontenc}

\usepackage{ifthen,calc}
\usepackage{babel}
\usepackage{xspace}
\usepackage{amsmath}

\usepackage{lmodern}
\usepackage{amsthm}

\numberwithin{equation}{section}
\theoremstyle{plain}
\newtheorem{theorem}{Theorem}[section]
\newtheorem{lemma}[theorem]{Lemma}
\newtheorem{corollary}[theorem]{Corollary}

\theoremstyle{remark}
\newtheorem{remark}[theorem]{Remark}
\theoremstyle{definition}

\DeclareMathOperator{\Div}{div}
\newcommand{\abs}[1]{\lvert#1\rvert}
\newcommand{\Abs}[1]{\left|#1\right|}

\newcommand{\norm}[1]{\lVert#1\rVert}

\newcommand{\norma}[2]{\norm{#1}_{#2}}

\newcommand{\SpDim}{N}
\newcommand{\numberspacefont}{\boldsymbol}
\newcommand{\R}{\numberspacefont{R}}
\newcommand{\RN}{\R^{\SpDim}}

\newcommand{\Pposbase}[3][*]{\ifthenelse{\equal{#1}{*}}%
{(#2)_{#3}}{\left(#2\right)_{#3}}}
\newcommand{\ppos}[1]{\Pposbase[*]{#1}{+}}

\newcommand{\pneg}[1]{\Pposbase[*]{#1}{-}}

\newcommand{\di}{\,\textup{\textmd{d}}}

\newcommand{\unk}{u}
\newcommand{\unkii}{v}

\newcommand{\pder}[2]{\frac{\partial #1}{\partial #2}}

\newcommand{\bdr}[1]{\partial #1}

\newcommand{\emf}{G}
\newcommand{\eif}{\mathcal{A}}
\newcommand{\eef}{\mathcal{B}}
\newcommand{\sof}{\mathcal{S}}
\newcommand{\grad}{\operatorname{\nabla}}
\DeclareMathOperator{\supp}{supp}
\newcommand{\der}[2]{\frac{\di #1}{\di #2}}

\newcommand{\vol}{V}
\newcommand{\dnf}{\rho}
\newcommand{\dnp}{\psi}
\newcommand{\msr}{\mu}
\newcommand{\msd}{\mu_{\dnf}}
\newcommand{\ipf}{\omega}
\newcommand{\bbt}{\beta}
\newcommand{\bsf}{Z}
\begin{document}

\title
[Equations on inhomogeneous manifolds]{Long time behavior of solutions of degenerate parabolic equations with inhomogeneous density on manifolds}%
\author{Daniele Andreucci}
\address{Department of Basic and Applied Sciences for Engineering\\Sapienza University of Rome, Italy}
\email{daniele.andreucci@sbai.uniroma1.it}
\thanks{The first author is member of the Gruppo Nazionale
   per la Fisica Matematica (GNFM) of the Istituto Nazionale di Alta Matematica
   (INdAM)}
\author{Anatoli F. Tedeev}
\address{South Mathematical Institute of VSC RAS\\Vladikavkaz, Russian Federation}
\email{a\_tedeev@yahoo.com}
\thanks{The second author was supported by Sapienza Grant C26V17KBT3}



\date{\today}
\begin{abstract}
  We consider the Cauchy problem for doubly non-linear degenerate parabolic equations on Riemannian manifolds of infinite volume, or in $\RN$. The equation contains a weight function as a capacitary coefficient which we assume to decay at infinity. We connect the behavior of non-negative solutions to the interplay between such coefficient and the geometry of the manifold, obtaining,  in a suitable subcritical range, estimates of the vanishing rate for long times and of the finite speed of propagation. In supercritical ranges we obtain universal bounds and prove blow up in a finite time of the (initially bounded) support of solutions.
\end{abstract}

\maketitle


\section{Introduction}\label{s:intro}

\subsection{Statement of the problem and general assumptions}
\label{s:stat}
We consider the Cauchy problem
\begin{alignat}{2}
  \label{eq:pde}
  \dnf(x)
  \unk_{t}
  -
  \Div(
  \unk^{m-1}
  \abs{\grad \unk}^{p-2}
  \grad \unk
  )
  &=
  0
  \,,
  &\qquad&
  x\in M
  \,,
  t>0
  \,,
  \\
  \label{eq:initd}
  \unk(x,0)
  &=
  \unk_{0}(x)
  \,,
  &\qquad&
  x\in M
  \,.
\end{alignat}
Here $M$ is a Riemannian manifold of topological dimension $\SpDim$, with infinite volume.
We always assume we are in the degenerate case, that is
\begin{equation}
  \label{eq:parameters}
  \SpDim>p>1
  \,,
  \qquad
  p+m>3
  \,,
\end{equation}
and that $\unk\ge 0$.  The inhomogeneous density $\dnf$ is assumed to
be a globally bounded, strictly positive and nonincreasing function of the distance $d$
from a fixed point $x_{0}\in M$. With a slight abuse of notation we
still denote $\dnf(d(x,x_{0}))=\dnf(x)$. In the following all balls
$B_{R}\subset M$ are understood to be centered at $x_{0}$, and we
denote $d(x)=d(x,x_{0})$, $\vol(R)=\msr(B_{R})$.

Let us briefly explain the interest of this problem; in the case when
$M=\RN$ with the Euclidean metric, the first results on the
qualitative surprising properties of solutions to the porous media
equation with inhomogeneous density are due to
\cite{Kamin:Rosenau:1981}, \cite{Rosenau:Kamin:1982} (in cases reduced to dimension $1$). The interface
blow up in the same setting was discovered in
\cite{Kamin:Kersner:1993} and proved in \cite{Tedeev:2007} for a
general class of doubly degenerate parabolic equations.

In the Euclidean case, where we assume that
$\dnf(x)=(1+\abs{x})^{-\alpha}$, $x\in\RN$, for a given
$0<\alpha\le \SpDim$, the behavior of solutions depends sharply on the
interplay between the nonlinearities appearing in the
equation. Specifically, two different features concern us here: the
form taken by sup bounds for solutions, and the property of finite
speed of propagation (which is actually connected to conservation of
mass), see \cite{Tedeev:2007} for the following results; see also \cite{Dzagoeva:Tedeev:2018}.
\\
If $\alpha\le p$ one can prove sup estimates similar in spirit to
those valid for the standard doubly nonlinear equation with
coefficients independent of $x$, though different in the details of
functional dependence on the parameters of the problem. That is, a
decay as a negative power law of time, multiplied by a suitable power
of the initial mass. But, if $\alpha>p$ a universal bound holds true,
that is the initial mass does not appear in the estimate anymore.
\\
If the initial data is compactly supported, the evolution of the
support of the solution differs markedly in the case
$\alpha<\alpha_{*}$ and $\alpha>\alpha_{*}$, where
$\alpha_{*}\in(p,\SpDim)$ is an explicit threshold. In the subcritical
case the support is bounded for all times, and mass is conserved
accordingly. In the supercritical case both properties fail after a
finite time interval has elapsed.

A more detailed comparison with the Euclidean case is presented in
Subsection~\ref{s:exam} below. Before passing to our results, we quote
the following papers dealing with parabolic problems in the presence
of inhomogeneous density: \cite{Martynenko:Tedeev:2007},
\cite{Martynenko:Tedeev:2008} where blow up phenomena are
investigated; \cite{Reyes:Vazquez:2009},
\cite{Kamin:Reyes:Vazquez:2010} for an asymptotic expansion of the
solution of the porous media equation; \cite{Nieto:Reyes:2013},
\cite{Iagar:Sanchez:2013} where the critical case is dealt with.

The main goal of the present paper is to find a similar
characterization of the possible behavior of solutions in terms of the
density function $\dnf$, the nonlinearities in the equation, and of
course the Riemannian geometry of $M$. See also
\cite{Andreucci:Cirmi:Leonardi:Tedeev:2001} for the Euclidean case; we
employ the energy approach of
\cite{Andreucci:Tedeev:1999,Andreucci:Tedeev:2005,Andreucci:Tedeev:2015,Andreucci:Tedeev:2017}.
We prove new embedding results which we think are of independent
interest, besides allowing us to achieve the sought after
precise characterization of the solutions to our problem.

The geometry of $M$ enters our results via the nondecreasing isoperimetric function $g$ such that
\begin{equation}
  \label{eq:intro_iso}
  \abs{\bdr{U}}_{\SpDim-1}
  \ge
  g(\msr(U))
  \,,
  \quad
  \text{for all open bounded Lipschitz $U\subset M$.}
\end{equation}
Here $\msr$ denotes the Riemannian measure on $M$, and
$\abs{\cdot}_{\SpDim-1}$ the corresponding $(\SpDim-1)$-dimensional
Haussdorff measure. The properties of $g$ are encoded in the function
\begin{equation*}
  \ipf(v)
  =
  \frac{v^{\frac{\SpDim-1}{\SpDim}}}{g(v)}
  \,,
  \qquad
  v>0
  \,,
  \qquad
  \ipf(0)
  =
  \lim_{v\to 0+}
  \ipf(v)
  \,,
\end{equation*}
which we assume to be continuous and nondecreasing; in the Euclidean case $\ipf$ is constant. We also assume that
for all $R>0$, $\gamma>1$,
\begin{equation}
  \label{eq:intro_doub}
  \vol(2 R)
  \le
  C
  \vol(R)
  \,,
\end{equation}
for a suitable constant $C>1$. In some results we need the following natural assumption on $\ipf$, or
on $g$ which is the same:
\begin{equation}
  \label{eq:intro_iso_vol}
  g(\vol(R))
  \ge
  c
  \frac{\vol(R)}{R}
  \,,
  \quad
  \text{i.e.,}
  \quad
  \ipf(\vol(R))
  \le
  c^{-1}
  \frac{R}{\vol(R)^{\frac{1}{\SpDim}}}
  \,,
\end{equation}
for $R>0$, where $c>0$ is a given constant. In fact, one could see
that the converse to this inequality follows from the assumed
monotonic character of $\ipf$; thus \eqref{eq:intro_iso_vol} in
practice assumes the sharpness of such converse.  Finally we require
\begin{equation}
  \label{eq:intro_hardy}
  \int_{0}^{s}
  \frac{\di\tau}{\vol^{(-1)}(\tau)^{p}}
  \di\msr
  \le
  C
  \frac{s}{\vol^{(-1)}(s)^{p}}
  \,,
  \qquad
  s>0
  \,,
\end{equation}
which clearly places a restriction on $p$ depending on the growth of $\vol$.

The density function $\dnf$ is assumed to satisfy for all $R>0$
\begin{equation}
  \label{eq:intro_dnf_doub}
  \dnf(2R)
  \ge
  C^{-1}
  \dnf(R)
  \,,
\end{equation}
for a suitable $C>1$.

\begin{remark}
  \label{r:dnf}
  It follows without difficulty from our arguments that the radial
  character and the assumptions on $\dnf$ can be replaced by
  analogous statements on a radial function $\tilde \dnf$ such that
  \begin{equation*}
    c
    \tilde\dnf(x)
    \le
    \dnf(x)
    \le
    c^{-1}
    \tilde \dnf(x)
    \,,
    \qquad
    x\in M
    \,,
  \end{equation*}
  for a given $0<c<1$.
\end{remark}
All the assumptions stated so far will be understood in the following unless explicitly noted.

\subsection{Conservation of mass}
\label{s:mass}
Since $\dnf$ is globally bounded, the concept of weak solution is
standard. We need the following easy a~priori result. Note that it
holds regardless of other assumptions on the parameters, whenever standard finitely supported cutoff test functions can be used in the weak formulation (see the proof in Section~\ref{s:emb}).

We assume for our first results that
\begin{equation}
  \label{eq:bdd_supp}
  \supp \unk_{0}
  \subset
  B_{R_{0}}
  \,,
  \qquad
  \text{for a given $R_{0}>0$.}
\end{equation}

\begin{theorem}
  \label{t:mass}
  Let $\unk$ be a solution to \eqref{eq:pde}--\eqref{eq:initd}, with $\dnf\unk_{0}\in L^{1}(M)$ satisfying
  \eqref{eq:bdd_supp}. Assume that for $0<t<\bar t$
  \begin{equation}
    \label{eq:mass_n}
    \supp \unk(t)
    \subset
    B_{\bar R}
    \,,
  \end{equation}
  for some $\bar R>R_{0}$. Then
  \begin{equation}
    \label{eq:mass_nn}
    \norma{\unk(t)\dnf}{L^{1}(M)}
    =
    \norma{\unk_{0}\dnf}{L^{1}(M)}
    \,,
    \qquad
    0<t<\bar t
    \,.
  \end{equation}
\end{theorem}

\begin{remark}
  \label{r:mass_subcr}
  At least in the subcritical case of
  Subsection~\ref{s:intro_subcritical}, a solution $\unk$ to
  \eqref{eq:pde}--\eqref{eq:initd} can be obtained as limit of a
  sequence to Dirichlet problems with vanishing boundary data on
  $B_{R}$ with $R\to+\infty$. Since we can limit the $L^{1}(M)$ norm
  of each such approximation only in terms of the initial mass,
  passing to the limit we infer
  \begin{equation}
    \label{eq:mass_subcr_a}
    \norma{\unk(t)\dnf}{L^{1}(M)}
    \le
    \gamma
    \norma{\unk_{0}\dnf}{L^{1}(M)}
    \,,
    \qquad
    0<t<+\infty
    \,.
  \end{equation}
  Notice that this bound follows without assuming finite speed of
  propagation.
  \\
  However, known results \cite{Otto:1996} imply uniqueness in the
  class of solutions satisfying finite speed of propagation. Below we
  prove for the constructed solution exactly this property, so that
  our results apply to the unique such solution. Perhaps more general
  results of uniqueness follow from arguments similar to the ones
  quoted, but we do not dwell on this problem here.
\end{remark}

\subsection{The subcritical cases}
\label{s:intro_subcritical}

In this Subsection we gather results valid in subcritical cases, where however we consider two different notions of subcriticality, the first one being the increasing character of the function in \eqref{eq:fsp_sub_fn}, the second one being condition \eqref{eq:dnp_nond}. The latter is stronger in practice, see Subsection~\ref{s:exam}.

We give first our basic result about finite speed of propagation.
\begin{theorem}
  \label{t:fsp_sub}
  Let \eqref{eq:bdd_supp}, \eqref{eq:mass_subcr_a} be fulfilled. For
  any given $t>0$ we have that $\supp \unk(t)\subset B_{R}$ provided
  \begin{equation}
    \label{eq:fsp_sub_n}
    R^{p}
    \dnf(R)^{p+m-2}
    \msr(B_{R})^{p+m-3}
    \ge
    \gamma
    t
    \norma{\unk_{0}\dnf}{L^{1}(M)}^{p+m-3}
    \,,
  \end{equation}
  and $R\ge 4R_{0}$.
\end{theorem}

Next result follows immediately from Theorem~\ref{t:fsp_sub}.
\begin{corollary}
  \label{co:fsp_sub}
  Let \eqref{eq:bdd_supp}, \eqref{eq:mass_subcr_a} be fulfilled.
  Assume also that the function
  \begin{equation}
    \label{eq:fsp_sub_fn}
    R\mapsto
    R^{p}
    \dnf(R)^{p+m-2}
    \msr(B_{R})^{p+m-3}
    \,,
    \qquad
    R>\bar R
    \,,
  \end{equation}
  is strictly increasing for some $\bar R>0$, and it becomes unbounded
  as $R\to+\infty$. For large $t>0$ define $\bsf_{0}(t)$ as the solution
  of
  \begin{equation}
    \label{eq:fsp_sub_bsf}
    R^{p}
    \dnf(R)^{p+m-2}
    \msr(B_{R})^{p+m-3}
    =
    \gamma
    t
    \norma{\unk_{0}\dnf}{L^{1}(M)}^{p+m-3}
    \,,
  \end{equation}
  where $\gamma$ is the same as in \eqref{eq:fsp_sub_n}. Then
  $\supp \unk(t)\subset B_{\bsf_{0}(t)}$ for all large $t>0$.
\end{corollary}

Then we proceed to state a sup bound which assumes finite speed of propagation, and is independent of our results above.
We need the following property of $\dnf$
\begin{equation}
  \label{eq:divergent_dnf}
  \int_{B_{R}}
  \dnf(x)
  \di\msr
  \le
  C
  \msr(B_{R})
  \dnf(R)
  \,,
\end{equation}
for a suitable constant $C>0$. For example \eqref{eq:divergent_dnf}
rules out $\dnf$ which decay too fast.
We also define the function
\begin{equation}
  \label{eq:dnp}
  \dnp(s)
  =
  s^{p}
  \dnf(s)
  \,,
  \qquad
  s\ge 0
  \,,
\end{equation}
and assume that there exists $C\ge1$ such that
\begin{equation}
  \label{eq:dnp_nond}
  \dnp(s)\le C \dnp(t)
  \,,
  \qquad
  \text{for all $t>s>0$.}
\end{equation}
We need in the following that for given $C>0$, $0<\alpha<p$,
\begin{equation}
  \label{eq:intro_subcr}
  \dnf(cs)
  \le
  C
  c^{-\alpha}
  \dnf(s)
  \,,
  \qquad
  \text{for all $s>0$ and $1>c>0$.}
\end{equation}
Essentially \eqref{eq:intro_subcr} implies that $\dnf(s)$ decays no faster than $s^{-\alpha}$ as $s\to+\infty$.

\begin{theorem}
  \label{t:sub_sup}
  Let \eqref{eq:divergent_dnf}--\eqref{eq:intro_subcr} be fulfilled, and assume \eqref{eq:bdd_supp}.
  Assume also that $\unk$ is a solution to \eqref{eq:pde}--\eqref{eq:initd}, satisfying
  \begin{equation}
    \label{eq:sub_sup_m}
    \supp \unk(t)
    \subset
    B_{\bsf(t)}
    \,,
    \qquad
    t>0
    \,,
  \end{equation}
  for a positive nondecreasing $\bsf\in C([0,+\infty))$. Then
  \begin{equation}
    \label{eq:sub_sup_n}
    \norma{\unk(t)}{L^{\infty}(M)}
    \le
    \gamma
    \Big(
    \frac{
      \bsf(t)^{p}
      \dnf(\bsf(t))
    }{
      t
    }
    \Big)^{\frac{1}{p+m-3}}
    \,,
    \qquad
    t>0
    \,.
  \end{equation}
\end{theorem}

Clearly, we can combine Corollary~\ref{co:fsp_sub} and Theorem~\ref{t:sub_sup} to infer an explicit sup bound for $\unk$.
\begin{theorem}
  \label{t:sub_sup_2}
  Let the assumptions of Corollary~\ref{co:fsp_sub} and of Theorem~\ref{t:sub_sup} be fulfilled.
  Then \eqref{eq:sub_sup_n} holds true for large $t$ with $\bsf$ replaced by $\bsf_{0}$ as in Corollary~\ref{co:fsp_sub}.
\end{theorem}

\subsection{The supercritical cases}

We drop in our first result below the assumption that $\unk_{0}$ be of bounded support.

\begin{theorem}
  \label{t:euc_sup}
  Let the metric in $M$ be Euclidean, i.e., $\ipf$ be constant. Assume
  that that $\dnp_{\alpha}(s)=s^{\alpha}\dnf(s)$ is nonincreasing for
  $s>s_{0}$, for some given $s_{0}>0$, $\SpDim\ge\alpha> p$. Let
  $\dnf\unk_{0}\in L^{1}(M)$, $\unk_{0}\ge 0$. Then
  \begin{equation}
    \label{eq:euc_sup_n}
    \norma{\unk(t)}{L^{\infty}(M)}
    \le
    \gamma
    t^{-\frac{1}{p+m-3}}
    \,,
    \qquad
    t>0
    \,.
  \end{equation}
\end{theorem}

\begin{theorem}
  \label{t:ibu}
  Let
  $\unk_{0}\in L^{1}(M)$ with bounded support, and assume that for
  some $\theta>0$
  \begin{gather}
    \label{eq:ibu_m}
    \int_{M}
    d(x)^{\frac{p}{p+m+\theta-3}}
    \dnf(x)^{\frac{p+m+\theta-2}{p+m+\theta-3}}
    \di\msr
    <
    +\infty
    \,,
    \\
    \label{eq:ibu_n}
    \int_{M}
    d(x)^{\frac{p(1+\theta)}{p+m-3}}
    \dnf(x)^{\frac{p+m+\theta-2}{p+m-3}}
    \di\msr
    <
    +\infty
    \,.
  \end{gather}
  Let $\unk$ be a solution to
  \eqref{eq:pde}--\eqref{eq:initd}.
  Then the law of conservation of mass and the boundedness
  of the support of $\unk(t)$ fail over $(0,\bar t)$ for
  a sufficiently large $\bar t>0$.
\end{theorem}

\begin{remark}
  \label{r:ibu}
  If $\dnp(s)=s^{p}\dnf(s)$ is bounded, then
  \begin{equation*}
    \dnp(s)^{\frac{1}{p+m+\theta-3}}
    \dnf(s)
    \ge
    \gamma_{0}
    \dnp(s)^{\frac{1+\theta}{p+m-3}}
    \dnf(s)
    \,,
  \end{equation*}
  so that in this case \eqref{eq:ibu_m} implies \eqref{eq:ibu_n}.
\end{remark}

\subsection{Examples}
\label{s:exam}
The simplest case is probably the one where $M=\RN$,
$\dnf(x)=(1+\abs{x})^{-\alpha}$, $\alpha\ge 0$. It is easily seen that
our general assumptions of Subsection~\ref{s:stat} are satisfied. Let
us state the conditions corresponding to the ones in our main results.
\\
The subcritical case where $\dnp$ is nondecreasing and
\eqref{eq:dnp_nond} holds true corresponds to $\alpha\le p$.
\\
The function in \eqref{eq:fsp_sub_fn} giving the correct finite speed
of propagation is strictly increasing to $+\infty$ as required in
Corollary~\ref{co:fsp_sub} since this condition corresponds to
\begin{equation}
  \label{eq:exam_euc_i}
  \alpha
  <
  \alpha_{*}
  :=
  \frac{
    \SpDim(p+m-3) + p
  }{
    p+m-2
  }
  \,,
\end{equation}
and $\SpDim>\alpha_{*}>p$ according to our restriction $p<\SpDim$. Furthermore,
\begin{equation}
  \label{eq:exam_euc_ii}
  \bsf_{0}(t)
  =
  \gamma
  (
  t
  \norma{\unk_{0}\dnf}{L^{1}(M)}^{p+m-3}
  )^{\frac{1}{(\SpDim-\alpha)(p+m-3)+p-\alpha}}
  \,.
\end{equation}
\\
Finally the subcritical sup estimate can be proved under condition
\eqref{eq:intro_subcr} which clearly corresponds to $\alpha<p$; it reads
\begin{equation}
  \label{eq:exam_euc_iii}
  \norma{\unk(t)}{L^{\infty}(M)}
  \le
  \gamma
  t^{-\frac{\SpDim-\alpha}{(\SpDim-\alpha)(p+m-3)+p-\alpha}}
  \norma{\unk_{0}\dnf}{L^{1}(M)}^{\frac{p-\alpha}{(\SpDim-\alpha)(p+m-3)+p-\alpha}}
  \,.
\end{equation}
\\
The supercritical sup estimate of Theorem~\ref{t:euc_sup} corresponds to
$\SpDim\ge \alpha>p$.
\\
The assumptions needed for interface blow up i.e.,
\eqref{eq:ibu_m}--\eqref{eq:ibu_n} correspond to $\SpDim\ge\alpha>\alpha_{*}$.

Other examples may be obtained essentially as revolution surfaces in the spirit of \cite{Andreucci:Tedeev:2015}.

\begin{remark}
  \label{r:local_coord}
  In local coordinates, denoted by $x^{i}$, the divergence term in the equation \eqref{eq:pde} is written as
  \begin{equation*}
    \frac{
      1
    }{
      \sqrt{\det(g_{ij})}
    }
    \sum_{i,j=1}^{\SpDim}
    \pder{}{{x^{i}}}
    \Big(
    \sqrt{\det(g_{ij})}
    g^{ij}
    \unk^{m-1}
    \abs{\grad\unk}^{p-2}
    \pder{\unk}{x^{j}}
    \Big)
    \,,
  \end{equation*}
  where $(g_{ij})$ denotes the Riemannian metric, $(g^{ij})=(g_{ij})^{-1}$ so that $\di\msr=\sqrt{\det{(g_{ij})}}\di x$, and
  \begin{equation*}
    \abs{\grad u}^{2}
    =
    \sum_{i,j=1}^{\SpDim}
    g^{ij}
    \pder{\unk}{{x^{j}}}
    \pder{\unk}{{x^{i}}}
    \,.
  \end{equation*}
\end{remark}

\subsection{Plan of the paper}
\label{s:intro_plan}

We prove in Section~\ref{s:emb} several necessary auxiliary results. In Subsection~\ref{s:emb_gen} we present some embeddings which are not used in the following, but which may be of independent interest. In Section~\ref{s:sub} we prove the results concerning the subcritical case, in Section~\ref{s:euc} we prove Theorem~\ref{t:euc_sup} dealing with the case of the Euclidean metric, and finally in Section~\ref{s:ibu} Theorem~\ref{t:ibu} about interface blow up.

\section{Embeddings}
\label{s:emb}

Let us note that, since $g$ is nondecreasing,
\begin{equation}
  \label{eq:intro_ipf_doub}
  \ipf(\gamma v)
  \le
  \gamma^{\frac{\SpDim-1}{\SpDim}}
  \ipf(v)
  \,,
  \qquad
  v>0
  \,,
  \gamma>1
  \,.
\end{equation}
This property will be used without further notice. We also employ throughout the notation
\begin{equation}
  \label{eq:bbl_meas}
  \bbt
  =
  \SpDim(p+m-3)+p
  \,,
  \quad
  \msr_{\dnf}(I)
  =
  \int_{I}
  \dnf
  \di\msr
  \,,
\end{equation}
for all measurable $I\subset M$.

\subsection{Embeddings involving $\ipf$}
\label{s:emb_ipf}
We begin with one of our main tools; actually an analogous embedding
was proved in \cite{Tedeev:1993} in the Euclidean setting. A proof in
our setting may follow \cite{Andreucci:Tedeev:2000} (where again the
setting was different); we sketch here the proof of the case we need,
for the reader's convenience.
\begin{lemma}
  \label{l:emb_old}
  Let $\unk\in W^{1,p}(M)$, $0<r<q\le \SpDim p/(\SpDim-p)$. Then
  \begin{equation}
    \label{eq:emb_old_n}
    \int_{M}
    \abs{\unk}^{q}
    \di\msr
    \le
    \gamma
    \ipf(E)^{q}
    E^{1+\frac{q}{\SpDim}-\frac{q}{p}}
    \norma{\grad \unk}{L^{p}(M)}^{q}
    \,,
  \end{equation}
  where
  \begin{equation}
    \label{eq:emb_old_nn}
    E
    =
    \big(
    \int_{M}
    \abs{\unk}^{r}
    \di\msr
    \Big)^{\frac{q}{q-r}}
    \Big(
    \int_{M}
    \abs{\unk}^{q}
    \di\msr
    \Big)^{-\frac{r}{q-r}}
    \,.
  \end{equation}
\end{lemma}

\begin{proof}
  We confine ourselves to the case $q\le p$, which is the one of our interest here. 
  \\
  Introduce the standard rearrangement function
  \begin{equation*}
    \unk^{*}(s)
    =
    \inf\{\lambda \mid \msr_{\lambda}<s\}
    \,,
    \quad
    \msr_{\lambda}
    =
    \msr(\{x\in M \mid \abs{\unk(x)}> \lambda\})
    \,,
    \quad
    \lambda \ge 0
    \,.
  \end{equation*}
  Then write for convenience of notation
  \begin{equation*}
    E_{s}
    =
    \int_{M}
    \abs{\unk(x)}^{s}
    \di\msr
    \,,
    \qquad
    s>0
    \,.
  \end{equation*}
  We have for a $k>0$ to be selected presently
  \begin{multline}
    \label{eq:emb_old_i}
    E_{q}
    =
    \int_{0}^{\msr_{0}}
    \unk^{*}(s)^{q}
    \di s
    \le
    \gamma(q)
    \int_{0}^{\msr_{k}}
    (\unk^{*}(s)-k)^{q}
    \di s
    +
    \gamma(q)
    k^{q}
    \msr_{k}
    +
    \int_{\msr_{k}}^{\msr_{0}}
    \unk^{*}(s)^{q}
    \di s
    \\
    =:
    I_{1}+I_{2}+I_{3}
    \,.
  \end{multline}
  Next we invoke Chebychev inequality
  \begin{equation*}
    k^{r}
    \msr_{k}
    \le
    E_{r}
    \,,
  \end{equation*}
  to bound
  \begin{equation}
    \label{eq:emb_old_ii}
    I_{2}+I_{3}
    \le
    \gamma
    \msr_{k}^{1-\frac{q}{r}}
    E_{r}^{\frac{q}{r}}
    +
    k^{q-r}
    \int_{\msr_{k}}^{\msr_{0}}
    \unk^{*}(s)^{r}
    \di s
    \le
    \gamma
    \msr_{k}^{1-\frac{q}{r}}
    E_{r}^{\frac{q}{r}}
    =
    \frac{1}{2}
    E_{q}
    \,.
  \end{equation}
  The last equality in \eqref{eq:emb_old_ii} is our choice of
  $k$, which amounts to $\msr_{k}=\gamma E$.  
  Note that we may assume $\msr_{0}$ as large as necessary, by
  approximating $\unk$ while keeping all the involved integral quantities
  stable. Thus we can safely assume that such a value of $k$
  exists. Hence we absorb $I_{2}+I_{3}$ on the left hand side of
  \eqref{eq:emb_old_i}. We then reason as in \cite{Talenti:1976} to obtain
  \begin{multline}
    \label{eq:emb_old_iii}
    E_{q}
    \le
    \gamma
    \int_{0}^{\msr_{k}}
    (\unk^{*}(s)-k)^{q}
    \di s
    \le
    \gamma
    \msr_{k}^{1-\frac{q}{p}}
    \Big(
    \int_{0}^{\msr_{k}}
    (\unk^{*}(s)-k)^{p}
    \di s
    \Big)^{\frac{q}{p}}
    \\
    \le
    \gamma
    \msr_{k}^{1-\frac{q}{p}}
    \Big(
    \int_{0}^{\msr_{k}}
    [-\unk^{*}_{s}(s)]^{p}
    g(s)^{p}
    [s
    g(s)^{-1}]^{p}
    \di s
    \Big)^{\frac{q}{p}}
    \\
    \le
    \gamma
    \msr_{k}^{1-\frac{q}{p}}
    [\msr_{k}g(\msr_{k})^{-1}]^{q}
    \Big(
    \int_{M}
    \abs{\grad \unk}^{p}
    \di\msr
    \Big)^{\frac{q}{p}}
    \,.
  \end{multline}
  We have exploited here the fact that $t\mapsto tg(t)^{-1}$ is
  increasing as it follows from our assumption that $\ipf$ is nondecreasing.

  Finally \eqref{eq:emb_old_n} follows from \eqref{eq:emb_old_iii} and
  from our choice $\msr_{k}=\gamma E$.
\end{proof}

\begin{corollary}
  \label{co:emb_old_p}
  Let $\unk\in W^{1,p}(M)$ and $0<r<p$. Then
  \begin{equation}
    \label{eq:emb_old_p_n}
    \int_{M}
    \abs{\unk}^{p}
    \di\msr
    \le
    \gamma
    \ipf(\msr(\supp \unk))^{\frac{p}{1+r H}}
    \Big(
    \int_{M}
    \abs{\unk}^{r}
    \di\msr
    \Big)^{\frac{pH}{1+r H}}
    \Big(
    \int_{M}
    \abs{\grad \unk}^{p}
    \di\msr
    \Big)^{\frac{1}{1+r H}}
    \,,
  \end{equation}
  where
  \begin{equation*}
    H
    =
    \frac{p}{\SpDim(p-r)}
    \,.
  \end{equation*}
\end{corollary}

\begin{proof}
  We select $q=p$ in Lemma~\ref{l:emb_old}. The statement follows
  from an elementary computation, when we also bound by means of
  H\"older's inequality
  \begin{equation}
    \label{eq:emb_old_p_i}
    E
    \le
    \Big[
    \msr(\supp \unk)^{1-\frac{r}{p}}
    \Big(
    \int_{M}
    \abs{\unk}^{p}
    \di\msr
    \Big)^{\frac{r}{p}}
    \Big]^{\frac{p}{p-r}}
    \Big(
    \int_{M}
    \abs{\unk}^{p}
    \di\msr
    \Big)^{-\frac{r}{p-r}}
    =
    \msr(\supp \unk)
    \,.
  \end{equation}
\end{proof}

\begin{corollary}
  \label{co:emb_old_wg}
  Assume $\dnp(s)=s^{p}\dnf(s)$ satisfies \eqref{eq:dnp_nond}, and
  that $\unk\in W^{1,p}(M)$ has support of finite measure.
  Then for all $R>0$,
  \begin{equation}
    \label{eq:emb_old_wg_n}
    \int_{M}
    \abs{\unk}^{p}
    \dnf
    \di\msr
    \le
    \gamma
    \big(
    \dnp(R)
    +
    \dnf(R)
    \ipf(\msr(\supp \unk))^{p}
    \msr(\supp \unk)^{\frac{p}{\SpDim}}
    \big)
    \int_{M}
    \abs{\grad \unk}^{p}
    \di\msr
    \,.
  \end{equation}
\end{corollary}

\begin{proof}
  Let $R>0$ be fixed, and let $\zeta$ be a cutoff function in $B_{2R}$, with
  \begin{equation*}
    \zeta(x)
    =
    1
    \,,
    \quad
    x\in B_{R}
    \,;
    \qquad
    \abs{\grad \zeta}
    \le
    \gamma R^{-1}
    \,.
  \end{equation*}
  Then for $\unk$ as in the statement,
  \begin{equation}
    \label{eq:emb_old_wg_k}
    \int_{M}
    \abs{\unk}^{p}
    \dnf
    \di\msr
    \le
    2^{p-1}
    \int_{M}
    (
    \abs{\unk}
    \zeta
    )^{p}
    \dnf
    \di\msr
    +
    2^{p-1}
    \int_{M}
    (
    \abs{\unk}
    (1-\zeta)
    )^{p}
    \dnf
    \di\msr
    \,.
  \end{equation}
  Then we invoke \eqref{eq:dnp_nond} to infer
  \begin{multline}
    \label{eq:emb_old_wg_i}
    \int_{M}
    (
    \abs{\unk}
    \zeta
    )^{p}
    \dnf
    \di\msr
    \le
    C
    (2 R)^{p}
    \dnf(2R)
    \int_{M}
    \frac{(\abs{\unk}\zeta)^{p}}{d(x)^{p}}
    \di\msr
    \\
    \le
    \gamma
    \dnp(R)
    \int_{M}
    \abs{\grad (\unk\zeta)}^{p}
    \di\msr
    \,,
  \end{multline}
  where we have used Hardy's inequality \eqref{eq:hardy_n}.
  Next we bound for $v=\msr(\supp\unk)$
  \begin{multline}
    \label{eq:emb_old_wg_ii}
    \int_{M}
    (
    \abs{\unk}
    (1-\zeta)
    )^{p}
    \dnf
    \di\msr
    \le
    \dnf(R)
    \int_{M}
    (
    \abs{\unk}
    (1-\zeta)
    )^{p}
    \di\msr
    \\
    \le
    \gamma
    \dnf(R)
    \ipf(v)^{p}
    v^{\frac{p}{\SpDim}}
    \int_{M}
    \abs{\grad(\abs{\unk}(1-\zeta))}^{p}
    \di\msr
    \,,
  \end{multline}
  where we have used embedding \eqref{eq:emb_old_n} with $q=p$ as well as \eqref{eq:emb_old_p_i}.
  \\
  Note that, on appealing once more to Hardy's inequality \eqref{eq:hardy_n}, we prove
  \begin{multline}
    \label{eq:emb_old_wg_iii}
    \int_{M}
    \big[
    \abs{\grad (\unk\zeta)}^{p}
    +
    \abs{\grad(\abs{\unk}(1-\zeta))}^{p}
    \big]
    \di\msr
    \le
    \gamma
    \int_{M}
    \abs{\grad\unk}^{p}
    \di\msr
    +
    \gamma
    R^{-p}
    \int_{B_{2R}\setminus B_{R}}
    \abs{\unk}^{p}
    \di \msr
    \\
    \le
    \gamma
    \int_{M}
    \abs{\grad\unk}^{p}
    \di\msr
    +
    \gamma
    \int_{M}
    \frac{\abs{\unk}^{p}}{d(x)^{p}}
    \di \msr
    \le
    \gamma
    \int_{M}
    \abs{\grad\unk}^{p}
    \di\msr
    \,.
  \end{multline}
  On using \eqref{eq:emb_old_wg_iii} in \eqref{eq:emb_old_wg_i},
  \eqref{eq:emb_old_wg_ii} we finally get \eqref{eq:emb_old_wg_n}.
\end{proof}

Next Theorem is not used in the following, but it may be of independent interest.
\begin{theorem}
  \label{t:emb_ipf_euc}
  Let the metric in $M$ be Euclidean, i.e., $\ipf$ be constant. Then
  for all $\unk\in W^{1,p}(M)$ with support $\supp \unk$ of finite
  measure we have for all $R>0$
  \begin{equation}
    \label{eq:emb_ipf_euc_n}
    \int_{M}
    \abs{\unk}^{p}
    \dnf
    \di\msr
    \le
    \gamma
    \Big\{
    \dnf(R)^{\frac{\SpDim}{p}-1}
    \msd(\supp \unk)
    +
    \int_{B_{R}\cap\supp\unk}
    \dnf^{\frac{\SpDim}{p}}
    \di\msr
    \Big\}^{\frac{p}{\SpDim}}
    \int_{M}
    \abs{\grad \unk}^{p}
    \di\msr
    \,.
  \end{equation}
\end{theorem}

\begin{proof}
  We may assume $\unk\ge 0$ and split
  \begin{equation*}
    \int_{M}
    \unk^{p}
    \dnf
    \di\msr
    =
    \int_{B_{R}}
    \unk^{p}
    \dnf
    \di\msr
    +
    \int_{M\setminus B_{R}}
    \unk^{p}
    \dnf
    \di\msr
    =:
    I^{1}
    +
    I^{2}
    \,.
  \end{equation*}
  Next by the standard Euclidean Sobolev embedding, for $p^{*}=\SpDim p/(\SpDim-p)$,
  \begin{multline}
    \label{eq:emb_ipf_euc_i}
    I^{1}
    \le
    \Big(
    \int_{B_{R}}
    \unk^{p^{*}}
    \di\msr
    \Big)^{\frac{p}{p^{*}}}
    \Big(
    \int_{B_{R}\cap \supp\unk}
    \dnf^{\frac{\SpDim}{p}}
    \di\msr
    \Big)^{\frac{p}{\SpDim}}
    \\
    \le
    \gamma
    \Big(
    \int_{M}
    \abs{\grad \unk}^{p}
    \di\msr
    \Big)
    \Big(
    \int_{B_{R}\cap \supp\unk}
    \dnf^{\frac{\SpDim}{p}}
    \di\msr
    \Big)^{\frac{p}{\SpDim}}
    \,.
  \end{multline}
  Next by the same token
  \begin{multline}
    \label{eq:emb_ipf_euc_ii}
    I^{2}
    \le
    \Big(
    \int_{M\setminus B_{R}}
    \unk^{p^{*}}
    \di\msr
    \Big)^{\frac{p}{p^{*}}}
    \Big(
    \int_{\supp\unk\setminus B_{R}}
    \dnf^{\frac{\SpDim}{p}-1}
    \dnf
    \di\msr
    \Big)^{\frac{p}{\SpDim}}
    \\
    \le
    \gamma
    \Big(
    \int_{M}
    \abs{\grad \unk}^{p}
    \di\msr
    \Big)
    \dnf(R)^{1-\frac{p}{\SpDim}}
    \msd(\supp\unk)^{\frac{p}{\SpDim}}
    \,.
  \end{multline}
  Collecting the estimates above we obtain \eqref{eq:emb_ipf_euc_n}.
\end{proof}

\begin{theorem}
  \label{t:emb_ipf_euc_sup}
  Let the metric in $M$ be Euclidean, i.e., $\ipf$ be constant. Assume that 
  that $\dnp_{\alpha}(s)=s^{\alpha}\dnf(s)$ is nonincreasing for $s>s_{0}$, for some
  given $s_{0}>0$, $\SpDim\ge\alpha\ge p$. Then for $\unk\in W^{1,p}(M)$ and
  $p(\SpDim-\alpha)/(\SpDim-p)<p_{1}<\SpDim p/(\SpDim-p)$ we have
  \begin{equation}
    \label{eq:emb_ipf_euc_sup_n}
    \int_{M}
    \abs{\unk}^{p_{1}}
    \dnf
    \di\msr
    \le
    \gamma
    \Big(
    \int_{M}
    \abs{\grad \unk}^{p}
    \di\msr
    \Big)^{\frac{p_{1}}{p}}
    \,.
  \end{equation}
\end{theorem}

\begin{proof}
  First we remark that owing to our assumption on $\dnp$
  \begin{equation}
    \label{eq:emb_ipf_euc_sup_i}
    \int_{M}
    \dnf^{\frac{p^{*}}{p^{*}-p_{1}}}
    \di\msr
    <
    +\infty
    \,.
  \end{equation}
  Indeed for $d(x)\ge s_{0}$ we have
  \begin{equation*}
    \dnf(x)
    \le
    \dnp_{\alpha}(s_{0})
    d(x)^{-\alpha}
    \,,
  \end{equation*}
  and for $p_{1}$ as in the statement
  \begin{equation*}
    \frac{\alpha p^{*}}{p^{*}-p_{1}}
    >
    \SpDim
    \,,
  \end{equation*}
  as one can immediately check. Next we apply H\"older inequality
  \begin{equation*}
    \int_{M}
    \abs{\unk}^{p_{1}}
    \dnf
    \di\msr
    \le
    \Big(
    \int_{M}
    \abs{\unk}^{p^{*}}
    \di\msr
    \Big)^{\frac{p_{1}}{p^{*}}}
    \Big(
    \int_{M}
    \dnf^{\frac{p^{*}}{p^{*}-p_{1}}}
    \di\msr
    \Big)^{\frac{p^{*}-p_{1}}{p^{*}}}
    \le
    \gamma
    \Big(
    \int_{M}
    \abs{\unk}^{p^{*}}
    \di\msr
    \Big)^{\frac{p_{1}}{p^{*}}}
    \,.
  \end{equation*}
  Finally we apply Sobolev embedding to prove \eqref{eq:emb_ipf_euc_sup_n}.
\end{proof}

We conclude by proving Hardy inequality, which has been used above. Its proof of course does not rely on the previous results.
\begin{theorem}[Hardy inequality]
  \label{t:hardy}
  For any $\unk\in W^{1,p}(M)$ we have
  \begin{equation}
    \label{eq:hardy_n}
    \int_{M}
    \frac{\abs{\unk}^{p}}{d(x)^{p}}
    \di\msr
    \le
    \gamma(\SpDim,p)
    \int_{M}
    \abs{\grad \unk}^{p}
    \di\msr
    \,.
  \end{equation}
\end{theorem}

\begin{proof}
  We may assume $\unk\ge 0$. With the notation of Lemma~\ref{l:emb_old}, we have
  \begin{equation}
    \label{eq:hardy_i}
    \int_{M}
    \frac{\unk^{p}}{d(x)^{p}}
    \di\msr
    \le
    \int_{0}^{+\infty}
    \unk^{*}(s)^{p}
    [d(\cdot)^{-p}]^{*}(s)
    \di s
    \,.
  \end{equation}
  On the other hand
  \begin{equation*}
    \msr(\{d(x)^{-p}>\lambda\})
    =
    \msr(B_{\lambda^{-\frac{1}{p}}})
    =
    \vol(\lambda^{-\frac{1}{p}})
    \,.
  \end{equation*}
  Therefore \eqref{eq:hardy_i} gives on integrating by parts
  \begin{equation}
    \label{eq:hardy_ii}
    \int_{M}
    \frac{\unk^{p}}{d(x)^{p}}
    \di\msr
    \le
    \int_{0}^{+\infty}
    \frac{\unk^{*}(s)^{p}}{\vol^{(-1)}(s)^{p}}
    \di s
    =
    p
    \int_{0}^{+\infty}
    \unk^{*}(s)^{p-1}
    [-\unk_{s}^{*}(s)]
    \int_{0}^{s}
    \frac{\di\tau}{\vol^{(-1)}(\tau)^{p}}
    \di s
    \,.
  \end{equation}
  Next we apply our assumption \eqref{eq:intro_hardy} in \eqref{eq:hardy_ii} and
  after applying H\"older inequality we arrive at
  \begin{equation}
    \label{eq:hardy_iii}
    \int_{0}^{+\infty}
    \frac{\unk^{*}(s)^{p}}{\vol^{(-1)}(s)^{p}}
    \di s
    \le
    \gamma
    \Big(
    \int_{0}^{+\infty}
    \frac{\unk^{*}(s)^{p}}{\vol^{(-1)}(s)^{p}}
    \di s
    \Big)^{\frac{p-1}{p}}
    \Big(
    \int_{0}^{+\infty}
    [-\unk^{*}_{s}(s)]^{p}
    \frac{s^{p}}{\vol^{(-1)}(s)^{p}}
    \di s
    \Big)^{\frac{1}{p}}
    \,.
  \end{equation}
  This immediately yields when we invoke \eqref{eq:intro_iso_vol}
  \begin{multline}
    \label{eq:hardy_iv}
    \int_{0}^{+\infty}
    \frac{\unk^{*}(s)^{p}}{\vol^{(-1)}(s)^{p}}
    \di s
    \le
    \gamma
    \int_{0}^{+\infty}
    [-\unk^{*}_{s}(s)]^{p}
    \frac{s^{p}}{\vol^{(-1)}(s)^{p}}
    \di s
    \\
    \le
    \gamma
    \int_{0}^{+\infty}
    [-\unk^{*}_{s}(s)]^{p}
    g(s)^{p}
    \di s
    \le
    \gamma
    \int_{M}
    \abs{\grad\unk}^{p}
    \di\msr
    \,,
  \end{multline}
  that is \eqref{eq:hardy_n}, by Polya-Szego principle.
\end{proof}

\subsection{A general embedding}
\label{s:emb_gen}
The results of this Subsection seem to us to be of independent
interest. They follow from a more direct and sharper approach based on
\eqref{eq:emb}. However, they lead to formal complications which in
practice make their use in our approach prohibitive, though they may
be applicable in some special cases.

We start assuming
\begin{equation}
  \label{eq:emb}
  \int_{M}
  \emf(\abs{f(x)})
  \di \msr
  \le
  \emf\Big(
  \int_{M}
  \abs{\grad f(x)}
  \di \msr
  \Big)
  \,,
  \qquad
  f\in W^{1,1}(M)
  \,.
\end{equation}
Here $\emf:[0,+\infty)\to[0,+\infty)$ is a convex and increasing
function, with $\emf(0)=0$. We remark that formally $G$ is the inverse
function of the function $g$ introduced in \eqref{eq:intro_iso}, and that \eqref{eq:emb}
could be actually proved by arguments relying on isoperimetric
properties, under extra assumptions.

We assume in the following that $p>1$ is such that the Cauchy problem
\begin{equation}
  \label{eq:ode_eif}
  \emf(\eif(s))
  =
  \eif'(s)^{\frac{p}{p-1}}
  \,,
  \quad
  s>0
  \,;
  \qquad
  \eif(0)
  =
  0
  \,,
\end{equation}
has a maximal solution $\eif$ with $\eif(s)>0$, $\eif'(s)>0$ for $s>0$. Then we define
\begin{equation*}
  \eef(s)
  =
  \emf(\eif(\abs{s}^{\frac{1}{p}}))
  \,,
  \qquad
  s\in\R
  \,.
\end{equation*}
We also extend for notational simplicity $\eif$ to $\R$ as an even
function, so that $\eif(s)=\eif(\abs{s})$ for $s\in\R$.

We assume also that for some $C>1$
\begin{equation}
  \label{eq:emb_Ap}
  \frac{\eif(s)}{s}
  \le
  \eif'(s)
  \le
  C
  \frac{\eif(s)}{s}
  \,,
  \quad
  s>0
  \,.
\end{equation}
The first inequality in \eqref{eq:emb_Ap} follows from the convexity of $\eif$, which is in turn a simple consequence of its definition; we remark that the second inequality is satisfied e.g., if $s\mapsto \eif'(s)s^{-\alpha}$ is nonincreasing for some $\alpha>0$, with $C=\alpha+1$.
\\
We also assume that $\eef$ is convex.

\begin{remark}
  \label{r:emb_power}
  For example, in the standard Euclidean case of $\RN$ we have
  that the admissible $p$ are
  those in $(1,\SpDim)$ and
  \begin{gather*}
    \emf(s)
    =
    \gamma_{\SpDim}
    s^{\frac{\SpDim}{\SpDim-1}}
    \,,
    \quad
    \eif(s)
    =
    c(p,\SpDim)
    s^{p\frac{\SpDim-1}{\SpDim-p}}
    \,,
    \\
    \eef(s)
    =
    c_{1}(p,\SpDim)
    s^{\frac{\SpDim}{\SpDim-p}}
    \,,
    \quad
    s\ge 0
    \,.
  \end{gather*}
\end{remark}

For notational brevity we introduce the function
\begin{equation*}
  \sof(s)
  =
  \emf^{(-1)}(s)^{p}
  s^{-(p-1)}
  \,,
  \qquad
  s\ge 0
  \,.
\end{equation*}

\begin{lemma}
  \label{l:emb_p}
  Let $p>1$ be as above; then for $\unk\in W^{1,p}(M)$
  \begin{equation}
    \label{eq:emb_p}
    \sof
    \Big(
    \int_{M}
    \emf(\eif(\unk(x)))
    \di\msr
    \Big)
    \le
    \int_{M}
    \abs{\grad \unk(x)}^{p}
    \di\msr
    \,.
  \end{equation}
\end{lemma}

\begin{proof}
  Choose in \eqref{eq:emb} $f(x)=\eif(\unk(x))$ and obtain
  \begin{multline*}
    \int_{M}
    \emf(\eif(\unk(x)))
    \di \msr
    \le
    \emf\Big(
    \int_{M}
    \Abs{
      \eif'(\unk(x))
      \grad\unk(x)
    }
    \di\msr
    \Big)
    \\
    \le
    \emf\bigg(
    \Big(
    \int_{M}
    \abs{\grad \unk}^{p}
    \di\msr
    \Big)^{\frac{1}{p}}
    \,
    \Big(
    \int_{M}
    \abs{
      \eif'(\unk(x))
    }^{\frac{p}{p-1}}
    \di\msr
    \Big)^{\frac{p-1}{p}}
    \bigg)
    \,.
  \end{multline*}
  Then we use \eqref{eq:ode_eif} and apply $\emf^{(-1)}$ to get \eqref{eq:emb_p}.
\end{proof}

Note that according to the definitions above
\begin{equation*}
  \sof(\eef(s))
  =
  \eif(s^{\frac{1}{p}})^{p}
  \emf(\eif(s^{\frac{1}{p}}))^{-(p-1)}
  =
  \frac{
    \eif(s^{\frac{1}{p}})^{p}
  }{
    \eif'(s^{\frac{1}{p}})^{p}
  }
  \,,
\end{equation*}
whence we get, on invoking \eqref{eq:emb_Ap},
\begin{equation}
  \label{eq:emb_sofeef}
  C^{-p}
  s
  \le
  \sof(\eef(s))
  \le
  s
  \,,
  \qquad
  s\ge 0
  \,.
\end{equation}

Our next result should be considered as a Faber-Krahn inequality.
\begin{corollary}
  \label{co:emb_eef}
  If $\unk\in W^{1,p}(M)$ has bounded support then
  \begin{equation}
    \label{eq:eef_n}
    \int_{M}
    \unk^{p}
    \di\msr
    \le
    v
    \eef^{(-1)}
    \Big(
    v^{-1}
    \eef
    \Big(
    C^{p}
    \int_{M}
    \abs{\grad \unk}^{p}
    \di\msr
    \Big)
    \Big)
    \,,
  \end{equation}
  where $v=\abs{\supp \unk}$.
\end{corollary}

\begin{proof}
  Let $v=\abs{\supp \unk}$; we may assume $\unk\ge0$. We start with
  \begin{equation*}
      v^{-1}
      \int_{M}
      \eef(\unk^{p})
      \di\msr
      =
      v^{-1}
      \int_{M}
      \emf(\eif(\unk))
      \di\msr
      \le
      v^{-1}
      \sof^{(-1)}
      \Big(
      \int_{M}
      \abs{\grad\unk}^{p}
      \di\msr
      \Big)
      \,,
  \end{equation*}
  where we used \eqref{eq:emb_p}.
  
  Thus we obtain, also employing Jensen inequality and \eqref{eq:emb_sofeef},
  \begin{multline*}
    \int_{M}
    \unk^{p}
    \di\msr
    \le
    v
    \eef^{(-1)}
    \Big(
    v^{-1}
    \int_{M}
    \eef(\unk^{p})
    \di\msr
    \Big)
    \le
    \\
    v
    \eef^{(-1)}
    \Big(
    v^{-1}
    \sof^{(-1)}
    \Big(
    \int_{M}
    \abs{\grad \unk}^{p}
    \di\msr
    \Big)
    \Big)
    \le
    v
    \eef^{(-1)}
    \Big(
    v^{-1}
    \eef
    \Big(
    C^{p}
    \int_{M}
    \abs{\grad \unk}^{p}
    \di\msr
    \Big)
    \Big)
    \,.
  \end{multline*}
\end{proof}

Finally we prove a weighted version of our previous result.

\begin{corollary}
  \label{co:emb_wg}
  Let $\dnp$ be nondecreasing.
  Under the assumptions of Corollary~\ref{co:emb_eef}, we have for all
  $R>0$, setting $A^{R}=\supp \unk\setminus B_{R}$,
  \begin{multline}
    \label{eq:emb_wg_n}
    \int_{M}
    \abs{\unk(x)}^{p}
    \dnf(x)
    \di\msr
    \le
    \gamma
    \dnp(R)
    \int_{M}
    \abs{\grad \unk(x)}^{p}
    \di\msr
    \\
    +
    \msd(A^{R})
    \eef^{(-1)}
    \Big(
    \frac{\dnf(R)}{\msd{(A^{R})}}
    \eef
    \Big(
    C^{p}
    \int_{M}
    \abs{\grad\unk}^{p}
    \di\msr
    \Big)
    \Big)
    \,.
  \end{multline}
\end{corollary}

\begin{proof}
  Fix $R>0$ and assume without loss of generality that $\unk\ge 0$.
  Let us begin with estimating
  \begin{multline}
    \label{eq:emb_wg_i}
    \int_{B_{R}}
    \unk(x)^{p}
    \dnf(x)
    \di\msr
    =
    \int_{B_{R}}
    \unk(x)^{p}
    d(x)^{-p}
    \dnp(x)
    \di\msr
    \\
    \le
    \gamma
    \dnp(R)
    \int_{M}
    \unk(x)^{p}
    d(x)^{-p}
    \di\msr
    \le
    \gamma
    \dnp(R)
    \int_{M}
    \abs{\grad \unk(x)}^{p}
    \di\msr
    \,.
  \end{multline}
  Here we have exploited the monotonicity of $\dnp$ and Hardy's inequality \eqref{eq:hardy_n}.

  Next from Jensen inequality and the definition of $\eef$, as well as
  from its assumed convexity, we find
  \begin{equation}
    \label{eq:emb_wg_ii}
    \eef\Big(
    \frac{1}{\msd(A^{R})}
    \int_{A^{R}}
    \unk^{p}
    \dnf
    \di\msr
    \Big)
    \\
    \le
    \frac{1}{\msd(A^{R})}
    \int_{A^{R}}
    \emf(\eif(\unk))
    \dnf
    \di\msr
    =:
    J
    \,.
  \end{equation}
  Since $\dnf$ is nonincreasing we can bound by means of \eqref{eq:emb_p}
  \begin{equation}
    \label{eq:emb_wg_iii}
    J
    \le
    \frac{\dnf(R)}{\msd{(A^{R})}}
    \int_{M}
    \emf(\eif(\unk))
    \di\msr
    \\
    \le
    \frac{\dnf(R)}{\msd{(A^{R})}}
    \sof^{(-1)}
    \Big(
    \int_{M}
    \abs{\grad\unk}^{p}
    \di\msr
    \Big)
    \,.
  \end{equation}
  Finally \eqref{eq:emb_wg_n} follows from \eqref{eq:emb_wg_i}--\eqref{eq:emb_wg_iii}, and from applying $\eef^{(-1)}$ as in the proof of Corollary~\ref{co:emb_eef}, as well as from \eqref{eq:emb_sofeef}.
\end{proof}

\subsection{Caccioppoli inequality}

We'll use the following inequalities.

\begin{lemma}
  \label{l:cacc}
  Let $\unk$ be a solution of \eqref{eq:pde}--\eqref{eq:initd}, and
  let $\theta>0$, with $\theta> 2-m$ if $m<1$, $k>h>0$ be given. Let
  $\zeta\in C^{1}_{0}((0,+\infty))$, $0\le \zeta\le 1$. Then
  \begin{multline}
    \label{eq:cacc_n}
    \sup_{0<\tau<t}
    \int_{M}
    (\ppos{\unk-k}\zeta)^{1+\theta}
    \dnf
    \di\msr
    +
    \int_{0}^{t}
    \int_{M}
    \abs{\grad (\ppos{\unk-k}\zeta)^{\frac{p+m+\theta-2}{p}}}^{p}
    \di\msr
    \di\tau
    \\
    \le
    \gamma
    H(h,k)
    \int_{0}^{t}
    \int_{M}
    \abs{\zeta_{\tau}}
    \ppos{\unk-h}^{1+\theta}
    \dnf
    \di\msr
    \di\tau
    \,,
  \end{multline}
  provided the right hand side in \eqref{eq:cacc_n} is finite. Here $H(h,k)=(k/(k-h))^{\pneg{m-1}}$.
\end{lemma}

\begin{lemma}
  \label{l:cacc2}
  Let $\unk$ be a solution of \eqref{eq:pde}--\eqref{eq:initd}, and
  let $\theta\ge p-1$. Let
  $\zeta\in C^{1}_{0}(M)$, $0\le \zeta\le 1$. Then
  \begin{multline}
    \label{eq:cacc2_n}
    \sup_{0<\tau<t}
    \int_{M}
    (\unk\zeta)^{1+\theta}
    \dnf
    \di\msr
    +
    \int_{0}^{t}
    \int_{M}
    \abs{\grad (\unk\zeta)^{\frac{p+m+\theta-2}{p}}}^{p}
    \di\msr
    \di\tau
    \\
    \le
    \gamma
    \Big\{
    \int_{0}^{t}
    \int_{M}
    \abs{\grad \zeta}^{p}
    \unk^{p+m+\theta-2}
    \di\msr
    \di\tau
    +
    \int_{M}
    (\unk_{0}\zeta)^{1+\theta}
    \di\msr
    \Big\}
    \,,
  \end{multline}
  provided the right hand side in \eqref{eq:cacc2_n} is finite.
\end{lemma}

The proofs of lemmas \ref{l:cacc} and \ref{l:cacc2} are standard and we omit them.

\subsection{Proof of Theorem~\ref{t:mass}}
\label{s:proof_mass}
Take in \eqref{eq:pde} as testing function a standard cut off function $\zeta=\zeta(x)$ such that
$\zeta\in C^{1}_{0}(\RN)$, with $\zeta=1$ in $B_{R}$, $R>\bar R$. Thus
\begin{equation*}
  \int_{M}
  \unk(t)
  \dnf
  \di\msr
  =
  \int_{M}
  \unk_{0}
  \dnf
  \di\msr
  \,,
  \qquad
  0<t<\bar t
  \,,
\end{equation*}
since $\grad\zeta=0$ on the support of $\unk$.

\section{The subcritical case}
\label{s:sub}

\begin{proof}[Proof of Theorem~\ref{t:fsp_sub}]
For $R>0$ to be chosen, we introduce the sequence of increasing annuli
\begin{gather*}
  A_{n}
  =
  \{x\in M
  \mid
  R_{n}'<d(x)<R_{n}''
  \}
  \,,
  \\
  R_{n}'=\frac{R}{2}(1-\eta-\sigma+\sigma 2^{-n})
  \,,
  \quad
  R_{n}''=R(1+\eta+\sigma - \sigma 2^{-n})
  \,.
\end{gather*}
We assume that $\supp\unk_{0}\subset B_{R/4}$ and
$0<\eta,\sigma\le 1/4$. Thus $\unk_{0}=0$ on all $A_{n}$.
\\
Let us also set for a fixed $\theta>0$ as in Lemma~\ref{l:cacc2}
\begin{gather}
  \label{eq:sub_fsp_kk}
  \unkii
  =
  \unk^{\frac{p+m+\theta-2}{p}}
  \,,
  \qquad
  \unkii_{n}
  =
  (\unk \zeta_{n})^{\frac{p+m+\theta-2}{p}}
  \,,
  \\
  \label{eq:sub_fsp_kkk}
  r
  =
  \frac{p}{p+m+\theta-2}
  <
  p
  \,,
  \qquad
  s
  =
  r(1+\theta)
  \,,
\end{gather}
for a sequence of cutoff functions $\zeta_{n}\in C^{1}_{0}(A_{n+1})$ such that
\begin{equation*}
  0\le \zeta_{n}\le 1
  \,;
  \quad
  \zeta_{n}(x)=1
  \,,
  \quad
  x\in A_{n}
  \,;
  \quad
  \abs{\grad \zeta_{n}}
  \le
  \gamma
  2^{n}
  (\sigma R)^{-1}
  \,.
\end{equation*}
As a consequence of Lemma~\ref{l:cacc2} we have for $n\ge 0$
\begin{equation}
  \label{eq:sub_i}
  J_{n}
  =
  \sup_{0<\tau<t}
  \int_{M}
  \unkii_{n}^{s}
  \dnf
  \di\msr
  +
  \int_{0}^{t}
  \int_{M}
  \abs{\grad \unkii_{n}}^{p}
  \di\msr
  \di\tau
  \le
  \gamma
  2^{np}
  (\sigma R)^{-p}
  \int_{0}^{t}
  \int_{M}
  \unkii_{n+1}^{p}
  \di\msr
  \di\tau
  \,.
\end{equation}
Next we bound the right hand side of \eqref{eq:sub_i} by means of the
embedding in Corollary~\ref{co:emb_old_p}, of \eqref{eq:intro_doub} and of Young's inequality, to obtain
\begin{equation}
  \label{eq:sub_ii}
  \begin{split}
    J_{n}
    &\le
    \frac{
      \gamma
      2^{np}}{
      (\sigma R)^{p}
    }
    \int_{0}^{t}
    \ipf(\vol(R))^{\frac{p}{1+r H}}
    \Big(
    \int_{M}
    \unkii_{n+1}^{r}
    \di\msr
    \Big)^{\frac{pH}{1+r H}}
    \Big(
    \int_{M}
    \abs{\grad\unkii_{n+1}}^{p}
    \di\msr
    \Big)^{\frac{1}{1+r H}}
    \di\tau
    \\
    &\le
    \delta
    \int_{0}^{t}
    \int_{M}
    \abs{\grad\unkii_{n+1}}^{p}
    \di\msr
    \di\tau
    \\
    &\quad
    +
    \frac{
      \gamma(\delta)
      b^{n}
      t
    }{
      (\sigma R)^{p+\frac{\SpDim(p-r)}{r}}
    }
    \ipf(\vol(R))^{\frac{\SpDim(p-r)}{r}}
    \Big(
    \sup_{0<\tau<t}
    \int_{M}
    \unkii_{n+1}(\tau)^{r}
    \di\msr
    \Big)^{\frac{p}{r}}
    \,.
  \end{split}
\end{equation}
Here $b=2^{p+\SpDim(p-r)/r}$, and $\delta>0$ is to be chosen presently. Indeed, exploiting recursively
\eqref{eq:sub_i}, \eqref{eq:sub_ii} we find for $j\ge 1$
\begin{multline*}
  J_{0}
  \le
  \delta^{j}
  \int_{M}
  \grad{\unkii_{j}}^{p}
  \di\msr
  \di\tau
  \\
  +
  \big(
  \sum_{i=0}^{j-1}
  (\delta b)^{i}
  \big)
  \frac{
    \gamma(\delta)
    t
  }{
    (\sigma R)^{p+\frac{\SpDim(p-r)}{r}}
  }
  \ipf(\vol(R))^{\frac{\SpDim(p-r)}{r}}
  \Big(
  \sup_{0<\tau<t}
  \int_{A_{\infty}}
  \unkii(\tau)^{r}
  \di\msr
  \Big)^{\frac{p}{r}}
  \,,
\end{multline*}
where we have set
\begin{equation*}
  A_{\infty}
  =
  \Big\{x\in M
  \mid
  \frac{R}{2}(1-\eta-\sigma)
  <d(x)<
  R(1+\eta+\sigma)
  \Big\}
  \,.
\end{equation*}
Thus on choosing $\delta<b^{-1}$, we infer as $j\to+\infty$, switching
back to the notation $\unk^{1+\theta}=\unkii^{s}$, $\unk=\unkii^{r}$,
\begin{multline}
  \label{eq:sub_iv}
  \sup_{0<\tau<t}
  \int_{A_{0}}
  \unk(\tau)^{1+\theta}
  \dnf
  \di\msr
  \le
  \frac{
    \gamma
    t
  }{
    (\sigma R)^{p+\frac{\SpDim(p-r)}{r}}
  }
  \ipf(\vol(R))^{\frac{\SpDim(p-r)}{r}}
  \dnf(R)^{-\frac{p}{r}}
  \\
  \times
  \Big(
  \sup_{0<\tau<t}
  \int_{A_{\infty}}
  \unk(\tau)
  \dnf
  \di\msr
  \Big)^{\frac{p}{r}}
  \,.
\end{multline}
We have used here \eqref{eq:intro_dnf_doub} to bound
\begin{equation*}
  \dnf(x)
  \ge
  \dnf(R(1+\sigma+\eta))
  \ge
  \dnf(2R)
  \ge
  C^{-1}
  \dnf(R)
  \,,
  \qquad
  x\in A_{\infty}
  \,.
\end{equation*}
Define next a decreasing sequence of annuli
\begin{gather*}
  D_{n}
  =
  \{
  x\in M
  \mid
  R_{n}^{*}
  <d(x)<
  R_{n}^{**}
  \}
  \,,
  \\
  R_{n}^{*}
  =
  \frac{R}{2}
  (1-2^{-n-1})
  \,,
  \qquad
  R_{n}^{**}
  =
  R(1+2^{-n-1})
  \,.
\end{gather*}
We apply \eqref{eq:sub_iv} recursively with $A_{0}=D_{n+1}$,
$A_{\infty}=D_{n}$, $\sigma=\eta=2^{-n-2}$, $n\ge 0$, to obtain, when recalling our definitions of $r$ and of $\bbt$ in \eqref{eq:bbl_meas},
\begin{multline}
  \label{eq:sub_v}
  \sup_{0<\tau<t}
  \int_{D_{n+1}}
  \unk(\tau)^{1+\theta}
  \dnf
  \di\msr
  \le
  \frac{
    \gamma
    b^{n}
    t
  }{
    R^{\bbt+\SpDim\theta}
    \dnf(R)^{p+m+\theta-2}
  }
  \ipf(\vol(R))^{\SpDim(p+m+\theta-3)}
  \\
  \times
  \Big(
  \sup_{0<\tau<t}
  \int_{D_{n}}
  \unk(\tau)
  \dnf
  \di\msr
  \Big)^{p+m+\theta-2}
  \,,
\end{multline}
where $b$ is as above. From \eqref{eq:sub_v} we get after an application of H\"older's inequality
\begin{multline}
  \label{eq:sub_vi}
  Y_{n}
  :=
  \sup_{0<\tau<t}
  \int_{D_{n}}
  \unk(\tau)
  \dnf
  \di\msr
  \le
  \Big(
  \sup_{0<\tau<t}
  \int_{D_{n}}
  \unk(\tau)^{1+\theta}
  \dnf
  \di\msr
  \Big)^{\frac{1}{1+\theta}}
  \Big(
  \int_{D_{n}}
  \dnf
  \di\msr
  \Big)^{\frac{\theta}{1+\theta}}
  \\
  \le
  \gamma
  (\vol(R)
  \dnf(R))^{\frac{\theta}{1+\theta}}
  \Big(
  \sup_{0<\tau<t}
  \int_{D_{n}}
  \unk^{1+\theta}
  \dnf
  \di\msr
  \Big)^{\frac{1}{1+\theta}}
  \,,
\end{multline}
on invoking assumption \eqref{eq:intro_dnf_doub}. Next we collect
\eqref{eq:sub_v} and \eqref{eq:sub_vi} (written for $n+1$) to get
\begin{equation}
  \label{eq:sub_vii}
  Y_{n+1}
  \le
  \gamma
  b^{\frac{n}{1+\theta}}
  \Big\{
  \frac{
    t \ipf(\vol(R))^{\SpDim(p+m+\theta-3)}
    \vol(R)^{\theta}
  }{
    R^{\bbt+\SpDim\theta}
    \dnf(R)^{p+m-2}
  }
  \Big\}^{\frac{1}{1+\theta}}
  Y_{n}^{1+\frac{p+m-3}{1+\theta}}
  \,,
\end{equation}
for $n\ge 0$. It follows from \cite[Lemma~5.6 Ch.~II]{LSU} that
$Y_{n}\to 0$ provided
\begin{equation}
  \label{eq:sub_viii}
  \frac{
    t \ipf(\vol(R))^{\SpDim(p+m+\theta-3)}
    \vol(R)^{\theta}
  }{
    R^{\bbt+\SpDim\theta}
    \dnf(R)^{p+m-2}
  }
  Y_{0}^{p+m-3}
  \le
  \gamma_{0}
  \,.
\end{equation}
In turn, in view of the bound \eqref{eq:mass_subcr_a} and of our
assumption \eqref{eq:intro_iso_vol}, \eqref{eq:sub_viii} is implied by
\begin{equation}
  \label{eq:sub_fsp_condition}
  \frac{
    t   \norma{\unk_{0}\dnf}{L^{1}(M)}^{p+m-3}
  }{
    R^{p}
    \dnf(R)^{p+m-2}
    \vol(R)^{p+m-3}
  }
  \le
  \gamma_{0}
  \,.
\end{equation}
Note that according to the definition of $Y_{n}$ in practice we have
proved that $\unk(x,t)=0$ for $x\in M\setminus B_{R}$, if $R$
satisfies \eqref{eq:sub_fsp_condition}, and of course the condition $\supp\unk_{0}\subset B_{R/4}$
stated at the beginning of the proof; the sought after result follows immediately.
\end{proof}

\begin{proof}[Proof of Theorem~\ref{t:sub_sup}]
  For a $k>0$ to be selected, and a fixed $\theta>0$ as in Lemma~\ref{l:cacc}, define for $n\ge 0$
  \begin{gather}
    \label{eq:sub_sup_kk}
    \unkii
    =
    \unk^{\frac{p+m+\theta-2}{p}}
    \,,
    \qquad
    \unkii_{n}
    =
    \ppos{\unk-k_{n}}^{\frac{p+m+\theta-2}{p}}
    \,,
    \\
    \label{eq:sub_sup_kkk}
    k_{n}
    =
    k(1-2\sigma + \sigma 2^{-n})
    \,,
    \quad
    r
    =
    \frac{p}{p+m+\theta-2}
    \,,
    \quad
    s
    =
    r(1+\theta)
    <
    p
    \,.
  \end{gather}
  Here $\sigma\in(0,1/4]$.
  We also define for a fixed $t>0$ the decreasing sequence
  \begin{equation}
    \label{eq:sub_sup_tt}
    \tau_{n}
    =
    \frac{t}{2}
    (1-2\sigma+\sigma 2^{-n})
    \,,
    \qquad
    n\ge 0
    \,.
  \end{equation}
  We introduce the notation
  $G_{n}(\tau)=\supp \unkii_{n}(\tau)$, $0<\tau<t$. Note that according to
  our assumptions, we have $G_{n}(\tau)\subset B_{\bsf(\tau)}$.
  \\
  We begin by an application of H\"older's inequality and then of
  embedding \eqref{eq:emb_old_wg_n}, obtaining for $0<\tau<t$ the bound
  \begin{equation}
    \label{eq:sub_sup_i}
    \begin{split}
      \int_{M}
      \unkii_{n+1}(\tau)^{s}
      \dnf
      \di\msr
      &\le
      \Big(
      \int_{M}
      \unkii_{n+1}(\tau)^{p}
      \dnf
      \di\msr
      \Big)^{\frac{s}{p}}
      \msd(G_{n+1}(\tau))^{1-\frac{s}{p}}
      \\
      &\le
      \gamma
      \big\{
      \dnp(R)
      +
      \dnf(R)
      \ipf(\msr(G_{n+1}(\tau)))^{p}
      \msr(G_{n+1}(\tau))^{\frac{p}{\SpDim}}
      \big\}^{\frac{s}{p}}
      \\
      &\quad
      \times
      \msd(G_{n+1}(\tau))^{1-\frac{s}{p}}
      \Big(
      \int_{M}
      \abs{\grad \unkii_{n+1}(\tau)}^{p}
      \di\msr
      \Big)^{\frac{s}{p}}
      =:
      K_{1}
      \,.
    \end{split}
  \end{equation}
  Next we select $R=L_{n+1}(\tau)$ according to
  \begin{equation}
    \label{eq:sub_sup_L}
    L_{n+1}(\tau)
    :=
    \ipf(\vol(\bsf(t)))
    \msr(G_{n+1}(\tau))^{\frac{1}{\SpDim}}
    \le
    \gamma
    \bsf(t)
    \,,
  \end{equation}
  where the inequality follows from \eqref{eq:intro_iso_vol} and from
  $G_{n+1}(\tau)\subset B_{\bsf(\tau)}\subset B_{\bsf(t)}$.
  Then, according to the definition of $\dnp$, both the terms in brackets in \eqref{eq:sub_sup_i} can be bounded in the same way leading us to
  \begin{equation}
    \label{eq:sub_sup_ii}
    \begin{split}
      \{\dots\}^{\frac{s}{p}}
      &\le
      \gamma
      \dnf(L_{n+1}(\tau))^{\frac{s}{p}}
      \ipf(\vol(\bsf(t)))^{s}
      \msr(G_{n+1}(\tau))^{\frac{s}{\SpDim}}
      \\
      &\le
      \gamma
      \Big(
      \frac{
        \bsf(t)
      }{
        L_{n+1}(\tau)
      }
      \Big)^{\alpha\frac{s}{p}}
      \dnf(\bsf(t))^{\frac{s}{p}}
      \ipf(\vol(\bsf(t)))^{s}
      \msr(G_{n+1}(\tau))^{\frac{s}{\SpDim}}
      \,,
    \end{split}
  \end{equation}
  where we have used assumption \eqref{eq:intro_subcr}.
  In turn by definition of $L_{n+1}(\tau)$ and by $\alpha<p$ we have
  in \eqref{eq:sub_sup_ii}
  \begin{multline}
    \label{eq:sub_sup_iv}
    \frac{\msr(G_{n+1}(\tau))^{\frac{s}{\SpDim}}}{L_{n+1}(\tau)^{\alpha\frac{s}{p}}}
    =
    \frac{
      \msr(G_{n+1}(\tau))^{\frac{s}{\SpDim}(1-\frac{\alpha}{p})}
    }{
      \ipf(\vol(\bsf(t)))^{\alpha\frac{s}{p}}
    }
    \\
    \le
    \frac{
      \dnf(\bsf(t))^{-\frac{s}{\SpDim}(1-\frac{\alpha}{p})}
      \msd(G_{n+1}(\tau))^{\frac{s}{\SpDim}(1-\frac{\alpha}{p})}
    }{
      \ipf(\vol(\bsf(t)))^{\alpha\frac{s}{p}}
    }
    \,,
  \end{multline}
  where we have estimated, appealing again to $G_{n+1}(\tau)\subset B_{\bsf(t)}$,
  \begin{multline}
    \label{eq:sub_sup_iii}
    \msr(G_{n+1}(\tau))
    =
    \int_{G_{n+1}(\tau)}
    \di\msr
    \\
    \le
    \int_{G_{n+1}(\tau)}
    \dnf(x)
    \dnf(\bsf(t))^{-1}
    \di\msr
    =
    \dnf(\bsf(t))^{-1}
    \msd(G_{n+1}(\tau))
    \,.
  \end{multline}
  Thus, collecting \eqref{eq:sub_sup_i}--\eqref{eq:sub_sup_iii}, we get, integrating also over $\tau_{n+1}<\tau<t$,
  \begin{equation}
    \label{eq:sub_sup_j}
    \begin{split}
      &\int_{\tau_{n+1}}^{t}
      \int_{M}
      \unkii_{n+1}^{s}
      \dnf
      \di\msr
      \\
      &\quad
      \le
      \gamma
      \mathcal{F}(t)
      \int_{\tau_{n+1}}^{t}
      \msd(G_{n+1}(\tau))^{1-\frac{s}{p}+\frac{s}{\SpDim}(1-\frac{\alpha}{p})}
      \Big(
      \int_{M}
      \abs{\grad \unkii_{n+1}}^{p}
      \di\msr
      \Big)^{\frac{s}{p}}
      \di\tau
      \,,
    \end{split}
  \end{equation}
  with
  \begin{equation*}
    \mathcal{F}(t)
    =
    \ipf(\vol(\bsf(t)))^{s(1-\frac{\alpha}{p})}
    \dnf(\bsf(t))^{\frac{s}{p}-\frac{s}{\SpDim}(1-\frac{\alpha}{p})}
    \bsf(t)^{\alpha\frac{s}{p}}
    \,.
  \end{equation*}
  We use the above estimate together with a standard application of Caccioppoli inequality in
  Lemma~\ref{l:cacc}, and Young inequality arriving at
  \begin{equation}
    \label{eq:sub_sup_jj}
    \begin{split}
      &I_{n}:=
      \sup_{\tau_{n}<\tau<t}
      \int_{M}
      \unkii_{n}^{s}
      \dnf
      \di\msr
      +
      \int_{\tau_{n}}^{t}
      \int_{M}
      \abs{\grad \unkii_{n}}^{p}
      \di\msr
      \di\tau
      \le
      \gamma
      \frac{2^{\ell n}}{\sigma^{\ell} t}
      \int_{\tau_{n+1}}^{t}
      \int_{M}
      \unkii_{n+1}^{s}
      \dnf
      \di\msr
      \di\tau
      \\
      &\quad
      \le
      \delta
      \int_{\tau_{n+1}}^{t}
      \int_{M}
      \abs{\grad \unkii_{n+1}}^{p}
      \di\msr
      \di\tau
      \\
      &\qquad
      +
      \gamma
      \delta^{-\frac{s}{p-s}}
      b^{n}
      \sigma^{-\frac{\ell p}{p-s}}
      t^{1-\frac{p}{p-s}}
      \mathcal{F}(t)^{\frac{p}{p-s}}
      \sup_{\tau_{\infty}<\tau<t}
      \msd(G_{\infty}(\tau))^{1+\frac{ps}{\SpDim(p-s)}(1-\frac{\alpha}{p})}
      \,.
    \end{split}
  \end{equation}
  Here $\delta>0$ is to be chosen, $\ell=1+\pneg{m-1}$ and $b=2^{\ell p/(p-s)}$; we have set
  \begin{gather}
    \label{eq:sub_sup_ttt}
    \tau_{\infty}
    =
    \lim_{n\to+\infty}
    \tau_{n}
    =
    \frac{t}{2}
    (1-2\sigma)
    \,,
    \quad
    k_{\infty}
    =
    \lim_{n\to+\infty}
    k_{n}
    =
    k(1-2\sigma)
    \,,
    \\
    \label{eq:sub_sup_tttt}
    G_{\infty}(\tau)
    =
    \supp \ppos{\unk(\tau)- k_{\infty}}
    \,.
  \end{gather}
  We can iterate \eqref{eq:sub_sup_jj} obtaining for $j\ge 1$
  \begin{multline}
    \label{eq:sub_sup_jjj}
    I_{0}
    \le
    \delta^{j}
    \int_{\tau_{j}}^{t}
    \int_{M}
    \abs{\grad \unkii_{j}}^{p}
    \di\msr
    \di\tau
    \\
    +
    \gamma
    \delta^{-\frac{s}{p-s}}
    \sigma^{-\frac{\ell p}{p-s}}
    \Big(
    \sum_{i=0}^{j-1}
    (\delta b)^{i}
    \Big)
    t^{-\frac{s}{p-s}}
    \mathcal{F}(t)^{\frac{p}{p-s}}
    \sup_{\tau_{\infty}<\tau<t}
    \msd(G_{\infty}(\tau))^{1+\frac{ps}{\SpDim(p-s)}(1-\frac{\alpha}{p})}
    \,.
  \end{multline}
  We select $\delta<b^{-1}$ and let $j\to+\infty$, arriving at the
  basic estimate needed to start our second and last iterative
  process:
  \begin{multline}
    \label{eq:sub_sup_k}
    \sup_{\tau_{0}<\tau<t}
    \int_{M}
    \unkii_{0}^{s}
    \dnf
    \di\msr
    \le
    \gamma
    \sigma^{-\frac{\ell p}{p-s}}
    t^{-\frac{s}{p-s}}
    \mathcal{F}(t)^{\frac{p}{p-s}}
    \sup_{\tau_{\infty}<\tau<t}
    \msd(G_{\infty}(\tau))^{1+\frac{ps}{\SpDim(p-s)}(1-\frac{\alpha}{p})}
    \,.
  \end{multline}
  The iteration makes use of the following definitions
  \begin{gather}
    \label{eq:sub_sup_x}
    \tau_{n}'
    =
    \frac{t}{2}(1-2^{-n-1})
    \,,
    \qquad
    k_{n}'
    =
    k(1-2^{-n-1})
    \,,
    \\
    \label{eq:sub_sup_xx}
    H_{n}(\tau)
    =
    \supp\ppos{\unk(\tau)-k_{n}'}
    \,,
    \qquad
    Y_{n}
    =
    \sup_{\tau_{2n}'<\tau<t}
    \msd(H_{2n}(\tau))
    \,.
  \end{gather}
  We apply Chebychev inequality as well as \eqref{eq:sub_sup_k} with
  $\sigma=2^{-2n-2}$, to get, recalling the definitions of $\unkii_{n}$
  and of $s$,
  \begin{multline}
    \label{eq:sub_sup_LSU}
    Y_{n+1}
    \le
    (2^{-2n-3}k)^{-1-\theta}
    \sup_{\tau_{2n+1}'<\tau<t}
    \int_{M}
    \ppos{\unk-k_{2n+1}'}^{1+\theta}
    \dnf
    \di\msr
    \\
    \le
    \gamma
    b^{n}
    k^{-1-\theta}
    t^{-\frac{s}{p-s}}
    \mathcal{F}(t)^{\frac{p}{p-s}}
    Y_{n}^{1+\frac{ps}{\SpDim(p-s)}(1-\frac{\alpha}{p})}
    \,,
  \end{multline}
  for $b=4^{1+\theta+\ell p/(p-s)}$. Invoking next \cite[Lemma~5.6 Ch.~II]{LSU} we get
  that $Y_{n}\to 0$ as $n\to+\infty$ provided
  \begin{equation}
    \label{eq:sub_sup_LSU_ii}
    k^{-1-\theta}
    t^{-\frac{s}{p-s}}
    \mathcal{F}(t)^{\frac{p}{p-s}}
    Y_{0}^{\frac{ps}{\SpDim(p-s)}(1-\frac{\alpha}{p})}
    \le
    \gamma_{0}
    \,.
  \end{equation}
  We remark that this amounts to $\unk(x,t)\le k$, $x\in M$.
  \\
  Next we note that since $H_{0}(\tau)\subset B_{\bsf(t)}$, $0<\tau<t$, we have
  \begin{equation}
    \label{eq:sub_sup_LSU_iii}
    Y_{0}
    \le
    \int_{B_{\bsf(t)}}
    \dnf
    \di\msr
    \le
    \gamma
    \vol(\bsf(t))
    \dnf(\bsf(t))
    \,,
  \end{equation}
  according to assumption \eqref{eq:divergent_dnf}. Using \eqref{eq:sub_sup_LSU_iii} in \eqref{eq:sub_sup_LSU_ii}, together with $s/(p-s)=
  (1+\theta)/(p+m-3)$, we see that \eqref{eq:sub_sup_LSU_ii} is implied by
  \begin{multline}
    \label{eq:sub_sup_LSU_iv}
    k^{-1-\theta}
    t^{-\frac{1+\theta}{p+m-3}}
    \dnf(\bsf(t))^{\frac{1+\theta}{p+m-3}}
    \vol(\bsf(t))^{\frac{p(1+\theta)}{\SpDim(p+m-3)}(1-\frac{\alpha}{p})}
    \bsf(t)^{\alpha\frac{1+\theta}{p+m-3}}
    \\
    \times
    \ipf(\vol(\bsf(t)))^{\frac{p(1+\theta)}{p+m-3}(1-\frac{\alpha}{p})}
    \le
    \gamma_{0}
    \,.
  \end{multline}
  Finally we substitute \eqref{eq:intro_iso_vol} in \eqref{eq:sub_sup_LSU_iv} to transform it into
  \begin{equation}
    \label{eq:sub_sup_LSU_v}
    k^{-1}
    t^{-\frac{1}{p+m-3}}
    \dnf(\bsf(t))^{\frac{1}{p+m-3}}
    \bsf(t)^{\frac{p}{p+m-3}}
    \le
    \gamma_{0}
    \,,
  \end{equation}
  whence \eqref{eq:sub_sup_n}.
\end{proof}

\section{The case of the Euclidean metric}
\label{s:euc}

We use here the embedding in Theorem~\ref{t:emb_ipf_euc_sup}.

\begin{proof}[Proof of Theorem~\ref{t:euc_sup}]
  We use the notation introduced in \eqref{eq:sub_sup_tt},
  \eqref{eq:sub_sup_ttt}, \eqref{eq:sub_sup_tttt},
  \eqref{eq:sub_sup_x}, \eqref{eq:sub_sup_xx}.  Fix
  $p_{1}\in(p,\SpDim p/(\SpDim-p))$; the value of $p_{1}$ will not
  affect the functional form of the final estimate.  We have by
  H\"older inequality and by the embedding in
  \eqref{eq:emb_ipf_euc_sup_n}
  \begin{multline}
    \label{eq:euc_sup_i}
    \int_{M}
    \unkii_{n+1}(\tau)^{s}
    \dnf
    \di\msr
    \le
    \Big(
    \int_{M}
    \unkii_{n+1}(\tau)^{p_{1}}
    \dnf
    \di\msr
    \Big)^{\frac{s}{p_{1}}}
    \msd(G_{n+1}(\tau))^{1-\frac{s}{p_{1}}}
    \\
    \le
    \gamma
    \Big(
    \int_{M}
    \abs{\grad \unkii_{n+1}(\tau)}^{p}
    \di\msr
    \Big)^{\frac{s}{p}}
    \msd(G_{n+1}(\tau))^{1-\frac{s}{p_{1}}}
    \,.
  \end{multline}
  Then reasoning as in \eqref{eq:sub_sup_jj} we find
  \begin{equation}
    \label{eq:euc_sup_ii}
    \begin{split}
      &I_{n}:=
      \sup_{\tau_{n}<\tau<t}
      \int_{M}
      \unkii_{n}^{s}
      \dnf
      \di\msr
      +
      \int_{\tau_{n}}^{t}
      \int_{M}
      \abs{\grad \unkii_{n}}^{p}
      \di\msr
      \di\tau
      \le
      \gamma
      \frac{2^{\ell n}}{\sigma^{\ell} t}
      \int_{\tau_{n+1}}^{t}
      \int_{M}
      \unkii_{n+1}^{s}
      \dnf
      \di\msr
      \di\tau
      \\
      &\quad
      \le
      \delta
      \int_{\tau_{n+1}}^{t}
      \int_{M}
      \abs{\grad \unkii_{n+1}}^{p}
      \di\msr
      \di\tau
      \\
      &\qquad
      +
      \gamma
      \delta^{-\frac{s}{p-s}}
      b^{n}
      \sigma^{-\frac{\ell p}{p-s}}
      t^{1-\frac{p}{p-s}}
      \sup_{\tau_{\infty}<\tau<t}
      \msd(G_{\infty}(\tau))^{\frac{p(p_{1}-s)}{p_{1}(p-s)}}
      \,.
    \end{split}
  \end{equation}
  Here $\delta>0$ is to be chosen and $b=2^{\ell p/(p-s)}$.
  We remark that straightforward arguments yield
  \begin{equation}
    \label{eq:euc_sup_iii}
    \frac{p(p_{1}-s)}{p_{1}(p-s)}
    >
    1
    \,.
  \end{equation}
  After selecting $\delta$ suitably small the same iterative procedure
  as in \eqref{eq:sub_sup_jjj} leads us to
  \begin{equation}
    \label{eq:euc_sup_iv}
    \sup_{\tau_{0}<\tau<t}
    \int_{M}
    \unkii_{0}^{s}
    \dnf
    \di\msr
    \le
    \gamma
    \sigma^{-\frac{\ell p}{p-s}}
    t^{-\frac{s}{p-s}}
    \sup_{\tau_{\infty}<\tau<t}
    \msd(G_{\infty}(\tau))^{\frac{p(p_{1}-s)}{p_{1}(p-s)}}
    \,.
  \end{equation}
  As in \eqref{eq:sub_sup_LSU} we get
  \begin{multline}
    \label{eq:euc_sup_LSU}
    Y_{n+1}
    \le
    (2^{-2n-3}k)^{-1-\theta}
    \sup_{\tau_{2n+1}'<\tau<t}
    \int_{M}
    \ppos{\unk-k_{2n+1}'}^{1+\theta}
    \dnf
    \di\msr
    \\
    \le
    \gamma
    b^{n}
    k^{-1-\theta}
    t^{-\frac{s}{p-s}}
    Y_{n}^{\frac{p(p_{1}-s)}{p_{1}(p-s)}}
    \,,
  \end{multline}
  for $b=4^{1+\theta+\ell p/(p-s)}$; owing to \cite[Lemma~5.6 Ch.~II]{LSU}
  and taking into account the definition of $s$, we get that
  $Y_{n}\to 0$ as $n\to+\infty$ provided
  \begin{equation}
    \label{eq:euc_sup_LSU_ii}
    k^{-1-\theta}
    t^{-\frac{1+\theta}{p+m-3}}
    Y_{0}^{\frac{1+\theta}{p+m-3}\, \frac{p_{1}-p}{p_{1}}}
    \le
    \gamma_{0}
    \,.
  \end{equation}
  In order to bound $Y_{0}$ we appeal once more to Chebychev inequality to find for $q>0$
  \begin{equation*}
    Y_{0}
    =
    \sup_{\frac{t}{4}<\tau<t}
    \msd(\supp \ppos{\unk(\tau)-k/2})
    \le
    \Big(\frac{2}{k}\Big)^{q+1}
    \sup_{\frac{t}{4}<\tau<t}
    \int_{M}
    \unk(\tau)^{q+1}
    \dnf
    \di\msr
    \,.
  \end{equation*}
  Thus, on defining 
  \begin{equation*}
    E_{q+1}(\tau)
    =
    \int_{M}
    \unk(\tau)^{q+1}
    \dnf
    \di\msr
    \,,
  \end{equation*}
  we conclude for all $t>0$
  \begin{equation}
    \label{eq:euc_sup_LSU_iii}
    \norma{\unk(t)}{L^{\infty}(M)}
    \le
    \gamma
    t^{-
      \frac{
        p_{1}
      }{
        p_{1}(p+m-3)+(p_{1}-p)(q+1)
      }
    }
    \sup_{\frac{t}{4}<\tau<t}
    E_{q+1}(\tau)^{
        \frac{
        p_{1}-p
      }{
        p_{1}(p+m-3)+(p_{1}-p)(q+1)
      }
    }
    \,.
  \end{equation}
  We are left with the task of estimating $E_{q+1}(\tau)$. We select $q>0$ large enough to have
  \begin{equation}
    \label{eq:euc_sup_E_i}
    \frac{p(\SpDim-\alpha)}{\SpDim-p}
    <
    p_{1}'
    :=
    \frac{p(1+q)}{p+m+q-2}
    <
    \frac{\SpDim p}{\SpDim-p}
    \,;
  \end{equation}
  indeed the leftmost side of \eqref{eq:euc_sup_E_i} is less than $p$
  since in our assumptions $\alpha>p$, while $(1+q)<p+m+q-2$ since $p+m>3$.
  Then from \eqref{eq:pde} we get for $\unkii=\unk^{(p+m+q-2)/p}$ the equality in
  \begin{multline}
    \label{eq:euc_sup_E_ii}
    \frac{1}{q+1}
    \der{E_{q+1}}{t}
    =
    -
    \Big(
    \frac{p}{p+m+q-2}
    \Big)^{p}
    \int_{M}
    \abs{\grad \unkii}^{p}
    \di\msr
    \\
    \le
    -
    \gamma
    \Big(
    \int_{M}
    \unk^{q+1}
    \dnf
    \di\msr
    \Big)^{\frac{p+m+q-2}{1+q}}
    \,,
  \end{multline}
  where the inequality follows from an application of
  Theorem~\ref{t:emb_ipf_euc_sup} with $p_{1}=p_{1}'$ as in
  \eqref{eq:euc_sup_E_i}. A direct integration gives
  \begin{equation}
    \label{eq:euc_sup_E_iii}
    E_{q+1}(t)
    \le
    \gamma
    t^{-\frac{1+q}{p+m-3}}
    \,,
    \qquad
    t>0
    \,;
  \end{equation}
  actually we integrate over $(t_{0},t)$ and then let $t_{0}\to 0+$,
  to circumvent possible problems with the local summability of the
  initial data.
  \\
  Finally we substitute \eqref{eq:euc_sup_E_iii} in
  \eqref{eq:euc_sup_LSU_iii} and arrive at the sought after estimate.
\end{proof}

\section{Interface blow up}
\label{s:ibu}

\begin{proof}[Proof of Theorem~\ref{t:ibu}]
  We assume by contradiction that $\unk(t)$ is compactly supported for all $t>0$.
  \\
  Let us compute by H\"older and Hardy inequalities
  \begin{equation}
    \label{eq:ibu_i}
    \begin{split}
      \int_{M}
      \unk
      \dnf
      \di\msr
      &
      \le
      \Big(
      \int_{M}
      d(x)^{-p}
      \unk^{p+m+\theta-2}
      \di\msr
      \Big)^{\frac{1}{p+m+\theta-2}}
      I(\theta)^{\frac{p+m+\theta-3}{p+m+\theta-2}}
      \\
      &\le
      \gamma
      \Big(
      \int_{M}
      \abs{\grad \unk^{\frac{p+m+\theta-2}{p}}}^{p}
      \di\msr
      \Big)^{\frac{1}{p+m+\theta-2}}
      I(\theta)^{\frac{p+m+\theta-3}{p+m+\theta-2}}
      \,,
    \end{split}
    \end{equation}
    where our assumption \eqref{eq:ibu_m} implies
    \begin{equation*}
      I(\theta)
      =
      \int_{M}
      d(x)^{\frac{p}{p+m+\theta-3}}
      \dnf(x)^{\frac{p+m+\theta-2}{p+m+\theta-3}}
      \di\msr
      <
      +\infty
      \,.
    \end{equation*}
    In a similar fashion
    \begin{equation}
    \label{eq:ibu_ii}
    \begin{split}
      \int_{M}
      \unk^{1+\theta}
      \dnf
      \di\msr
      &
      \le
      \Big(
      \int_{M}
      d(x)^{-p}
      \unk^{p+m+\theta-2}
      \di\msr
      \Big)^{\frac{1+\theta}{p+m+\theta-2}}
      J(\theta)^{\frac{p+m-3}{p+m+\theta-2}}
      \\
      &\le
      \gamma
      \Big(
      \int_{M}
      \abs{\grad \unk^{\frac{p+m+\theta-2}{p}}}^{p}
      \di\msr
      \Big)^{\frac{1+\theta}{p+m+\theta-2}}
      J(\theta)^{\frac{p+m-3}{p+m+\theta-2}}
      \,,
    \end{split}
    \end{equation}
    where again according to our assumption \eqref{eq:ibu_n} for suitable $\theta>0$
    \begin{equation*}
      J(\theta)
      =
      \int_{M}
      d(x)^{\frac{p(1+\theta)}{p+m-3}}
      \dnf(x)^{\frac{p+m+\theta-2}{p+m-3}}
      \di\msr
      <
      +\infty
      \,.
    \end{equation*}
    On using \eqref{eq:ibu_ii} and the equation \eqref{eq:pde}, we prove that, for $\unkii=\unk^{(p+m+\theta-2)/p}$,
    \begin{multline}
      \label{eq:ibu_k}
      \frac{1}{\theta+1}
      \der{}{t}
      \int_{M}
      \unk^{1+\theta}
      \dnf
      \di\msr
      =
      -
      \Big(
      \frac{p}{p+m+\theta-2}
      \Big)^{p}
      \int_{M}
      \abs{\grad \unkii}^{p}
      \di\msr
      \\
      \le
      -
      \gamma
      \Big(
      \int_{M}
      \unk^{1+\theta}
      \dnf
      \di\msr
      \Big)^{\frac{p+m+\theta-2}{1+\theta}}
      \,.
    \end{multline}
    Note that here $\theta>0$ however is small enough for our
    assumptions \eqref{eq:ibu_n} to hold true. We integrate the last
    differential inequality to obtain
    \begin{equation}
      \label{eq:ibu_iii}
      \int_{M}
      \unk(t)^{1+\theta}
      \dnf
      \di\msr
      \le
      \gamma
      t^{-\frac{1+\theta}{p+m-3}}
      \,,
      \qquad
      t>0
      \,.
    \end{equation}
    However owing to \eqref{eq:ibu_i} and to an application of H\"older inequality
    \begin{equation}
      \label{eq:ibu_iv}
      \int_{t}^{t+1}
      \int_{M}
      \unk
      \dnf
      \di\msr
      \di\tau
      \le
      \gamma
      \Big(
      \int_{t}^{t+1}
      \int_{M}
      \abs{\grad \unk^{\frac{p+m+\theta-2}{p}}}^{p}
      \di\msr
      \di\tau
      \Big)^{\frac{1}{p+m+\theta-2}}
      \,.
    \end{equation}
    Again integrating the equality in \eqref{eq:ibu_k} we get
    \begin{equation}
      \label{eq:ibu_kk}
      \int_{t}^{t+1}
      \int_{M}
      \abs{\grad \unk^{\frac{p+m+\theta-2}{p}}}^{p}
      \di\msr
      \di\tau
      \le
      \gamma
      \int_{M}
      \unk(t)^{1+\theta}
      \dnf
      \di\msr
      \,,
    \end{equation}
    which combined with \eqref{eq:ibu_iii} yields finally
    \begin{equation}
      \label{eq:ibu_kkk}
      \int_{M}
      \unk_{0}
      \dnf
      \di\msr
      =
      \int_{t}^{t+1}
      \int_{M}
      \unk(\tau)
      \dnf
      \di\msr
      \di\tau
      \le
      \gamma
      t^{
        - \frac{1+\theta}{(p+m-3)(p+m+\theta-2)}
      }
      \,.
    \end{equation}
    Indeed, since we are assuming by contradiction that the support of
    the solution is bounded over $(0,t+1)$, and therefore conservation
    of mass takes place in the same interval, according to
    Theorem~\ref{t:mass}. But \eqref{eq:ibu_kkk} is clearly
    inconsistent when $t\to+\infty$, thereby proving our statement.
\end{proof}

\bibliographystyle{abbrv}
\bibliography{paraboli,pubblicazioni_andreucci}

\def\cprime{$'$}
\begin{thebibliography}{10}

\bibitem{Andreucci:Cirmi:Leonardi:Tedeev:2001}
D.~Andreucci, R.~Cirmi, S.~Leonardi, and A.~F. Tedeev.
\newblock Large time behavior of solutions to the {N}eumann problem for a
  quasilinear second order degenerate parabolic equation in domains with
  noncompact boundary.
\newblock {\em Journal of Differential Equations}, 174:253--288, 2001.
\newblock Elsevier.

\bibitem{Andreucci:Tedeev:1999}
D.~Andreucci and A.~F. Tedeev.
\newblock A {F}ujita type result for a degenerate {N}eumann problem in domains
  with non compact boundary.
\newblock {\em J. Math.\ Analysis and Appl.}, 231:543--567, 1999.
\newblock Elsevier.

\bibitem{Andreucci:Tedeev:2000}
D.~Andreucci and A.~F. Tedeev.
\newblock Sharp estimates and finite speed of propagation for a {N}eumann
  problem in domains narrowing at infinity.
\newblock {\em Advances Diff. Eqs.}, 5:833--860, 2000.
\newblock Khayyam Publ., Athens Ohio (U.S.A.).

\bibitem{Andreucci:Tedeev:2005}
D.~Andreucci and A.~F. Tedeev.
\newblock Universal bounds at the blow-up time for nonlinear parabolic
  equations.
\newblock {\em Advances in Differential Equations}, 10:89--120, 2005.
\newblock Khayyam Publ., Athens Ohio (U.S.A.).

\bibitem{Andreucci:Tedeev:2015}
D.~Andreucci and A.~F. Tedeev.
\newblock Optimal decay rate for degenerate parabolic equations on noncompact
  manifolds.
\newblock {\em Methods Appl. Anal.}, 22(4):359--376, 2015.

\bibitem{Andreucci:Tedeev:2017}
D.~Andreucci and A.~F. Tedeev.
\newblock Large time behavior for the porous medium equation with convection.
\newblock {\em Meccanica}, 52(13):3255--3260, 2017.
\newblock MR3709966.

\bibitem{Dzagoeva:Tedeev:2018}
L.~F. Dzagoeva and A.~F. Tedeev.
\newblock Asymptotic behavior of the solution of doubly degenerate parabolic
  equations with inhomogeneous density.
\newblock {\em Submitted}.

\bibitem{Iagar:Sanchez:2013}
R.~G. Iagar and A.~S\'{a}nchez.
\newblock Asymptotic behavior for the heat equation in nonhomogeneous media
  with critical density.
\newblock {\em Nonlinear Anal.}, 89:24--35, 2013.

\bibitem{Kamin:Kersner:1993}
S.~Kamin and R.~Kersner.
\newblock Disappearance of interfaces in finite time.
\newblock {\em Meccanica}, 28(2):117--120, Jun 1993.

\bibitem{Kamin:Reyes:Vazquez:2010}
S.~Kamin, G.~Reyes, and J.~L. V{\'a}zquez.
\newblock Long time behavior for the inhomogeneous {PME} in a medium with
  rapidly decaying density.
\newblock {\em Discrete Contin. Dyn. Syst.}, 26(2):521--549, 2010.

\bibitem{Kamin:Rosenau:1981}
S.~Kamin and P.~Rosenau.
\newblock Propagation of thermal waves in an inhomogeneous medium.
\newblock {\em Comm. Pure Appl. Math.}, 34(6):831--852, 1981.

\bibitem{LSU}
O.~A. Ladyzhenskaja, V.~A. Solonnikov, and N.~N. Ural'ceva.
\newblock {\em Linear and Quasilinear Equations of Parabolic Type}, volume~23
  of {\em Translations of Mathematical Monographs}.
\newblock American Mathematical Society, Providence, RI, 1968.

\bibitem{Martynenko:Tedeev:2007}
A.~V. Martynenko and A.~F. Tedeev.
\newblock The {C}auchy problem for a quasilinear parabolic equation with a
  source and nonhomogeneous density.
\newblock {\em Zh. Vychisl. Mat. Mat. Fiz.}, 47(2):245--255, 2007.

\bibitem{Martynenko:Tedeev:2008}
A.~V. Martynenko and A.~F. Tedeev.
\newblock On the behavior of solutions of the {C}auchy problem for a degenerate
  parabolic equation with nonhomogeneous density and a source.
\newblock {\em Zh. Vychisl. Mat. Mat. Fiz.}, 48(7):1214--1229, 2008.

\bibitem{Nieto:Reyes:2013}
S.~Nieto and G.~Reyes.
\newblock Asymptotic behavior of the solutions of the inhomogeneous porous
  medium equation with critical vanishing density.
\newblock {\em Commun. Pure Appl. Anal.}, 12(2):1123--1139, 2013.

\bibitem{Otto:1996}
F.~Otto.
\newblock {$L^1$}-contraction and uniqueness for quasilinear elliptic-parabolic
  equations.
\newblock {\em J. Differential Equations}, 131(1):20--38, 1996.

\bibitem{Reyes:Vazquez:2009}
G.~Reyes and J.~L. V\'{a}zquez.
\newblock Long time behavior for the inhomogeneous {PME} in a medium with
  slowly decaying density.
\newblock {\em Commun. Pure Appl. Anal.}, 8(2):493--508, 2009.

\bibitem{Rosenau:Kamin:1982}
P.~Rosenau and S.~Kamin.
\newblock Nonlinear diffusion in a finite mass medium.
\newblock {\em Comm. Pure Appl. Math.}, 35(1):113--127, 1982.

\bibitem{Talenti:1976}
G.~Talenti.
\newblock Elliptic equations and rearrangements.
\newblock {\em Annali Scuola Normale Superiore di Pisa}, 3:697--718, 1976.

\bibitem{Tedeev:1993}
A.~F. Tedeev.
\newblock Qualitative properties of solutions of {N}eumann problem for
  quasilinear higher order parabolic equations.
\newblock {\em Ukrainian Mat. Journal}, 44:1571--1579, 1993.

\bibitem{Tedeev:2007}
A.~F. Tedeev.
\newblock The interface blow-up phenomenon and local estimates for doubly
  degenerate parabolic equations.
\newblock {\em Appl. Anal.}, 86(6):755--782, 2007.

\end{thebibliography}
\end{document}